\documentclass[final,hidelinks,onefignum,onetabnum]{siamart220329}

\usepackage{lipsum}
\usepackage{amsfonts}
\usepackage{graphicx}
\usepackage{epstopdf} 
\ifpdf
\DeclareGraphicsExtensions{.eps,.pdf,.png,.jpg}
\else
\DeclareGraphicsExtensions{.eps}
\fi

\usepackage{booktabs}
\usepackage{hyperref}
\usepackage{stmaryrd}
\usepackage{bm}
\usepackage{caption}
\usepackage{subcaption}
\usepackage{multirow}
\usepackage{enumitem}

\usepackage{amsopn}

\usepackage{amsmath}
\usepackage{mathtools}
\usepackage{framed,multirow,booktabs}
\usepackage{threeparttable}

\newcommand{\bfF}{\mathbf{F}}

\newcommand{\bfM}{\mathbf{M}}

\usepackage{algorithm,algorithmicx,algpseudocode} 

\usepackage[left=2.48cm,right=5.96cm,top=3.29cm,bottom=3.26cm,asymmetric]{geometry}

\newsiamthm{test}{Test}
\newsiamthm{assumption}{Assumption}

\newsiamremark{remark}{Remark}
\newsiamremark{hypothesis}{Hypothesis}
\crefname{hypothesis}{Hypothesis}{Hypotheses}
\newsiamthm{claim}{Claim}
\newsiamremark{exmp}{Example}

\headers{}{}

\title{Provably realizability-preserving finite volume method for quadrature-based moment models of kinetic equations\thanks{This work was partially supported by Shenzhen Science and Technology Program (No.~RCJC20221008092757098), National Natural Science Foundation of China (No.~12171227), and the Deutsche Forschungsgemeinschaft (DFG, German Research Foundation) - SPP 2410 \textit{Hyperbolic Balance Laws in Fluid Mechanics: Complexity, Scales, Randomness (CoScaRa)}.}
}

\author{Chuan Fan\thanks{Department of Mathematics, Southern University of Science and Technology, Shenzhen, Guangdong 518055, China(\email{fanc@sustech.edu.cn}).} 
\and Qian Huang\thanks{Institute of Applied Analysis and Numerical Simulation, University of Stuttgart, Stuttgart, 70569, Germany,(\email{qian.huang@mathematik.uni-stuttgart.de; hqqh91@qq.com}).} 
\and Kailiang Wu\thanks{Corresponding author. Department of Mathematics and Shenzhen International Center for Mathematics, Southern University of Science and Technology, Shenzhen, Guangdong 518055, China (\email{wukl@sustech.edu.cn}).} 
} 

\ifpdf
\hypersetup{
	pdftitle={},
	pdfauthor={}
}
\fi

\usepackage{comment}

\begin{document}
	
\maketitle
\begin{abstract}	
Quadrature-based moment methods (QBMM) provide tractable closures for multiscale kinetic equations, with diverse applications across aerosols, sprays, and particulate flows, etc. However, for the derived hyperbolic moment-closure systems, seeking numerical schemes preserving \emph{moment realizability}
is essential yet challenging due to strong nonlinear coupling and the lack of explicit conservative-to-flux maps.  
This paper proposes and analyzes a provably realizability-preserving finite-volume method for five-moment systems closed by the two-node Gaussian-EQMOM and three-point HyQMOM. 
Rather than relying on kinetic fluxes, we recast the realizability condition into a nonnegative quadratic form in the moment vector, reducing the original nonlinear constraints to bilinear inequalities amenable to analysis. 
On this basis, we construct a tailored Harten--Lax--van Leer (HLL) flux with rigorously derived wave speeds and intermediate states that embed realizability directly into the flux evaluation. 
We prove sufficient realizability-preserving conditions under explicit Courant--Friedrichs--Lewy (CFL) constraints in the collisionless case, and for BGK relaxation, we obtain coupled time-step conditions involving a realizability radius; a semi-implicit BGK variant inherits the collisionless CFL.  
From a \emph{multiscale} perspective, the analysis yields stability conditions uniform in the relaxation time and supports stiff-to-kinetic transitions. A practical limiter enforces strict realizability of reconstructed interface states without degrading accuracy. Numerical experiments demonstrate the accuracy, robustness in low-density regions, and realizability for both closures. This framework unifies realizability preservation for solving hyperbolic moment systems with complex closures and extends naturally to higher-order space--time discretizations.
 
\end{abstract}
	
\begin{keywords}
multiscale kinetic equations,
moment method,
realizability,
finite volume method,
positivity preserving scheme 
\end{keywords}

\begin{MSCcodes}
	65M60, 65M12, 82B40, 82C40, 35L60
\end{MSCcodes}

\section{Introduction}
Kinetic theories are widely used in modeling diverse interacting systems such as rarefied gas dynamics \cite{grad1949kinetic,xu2010unified}, dispersed multiphase flows \cite{williams1958spray}, plasma physics \cite{lenard1960bogoliubov}, and stellar dynamics \cite{lynden1967statistical}. They describe the evolution of a problem-specific particle number density function (NDF), which depends on space-time coordinates, velocity, and possibly other intrinsic variables. Due to the high dimensionality, various model reduction approaches have been developed to reduce computational cost while retaining essential physics, including the BGK model \cite{BGK1954}, discrete velocity methods \cite{mieussens2000discrete}, and moment methods \cite{grad1949kinetic,levermore1996moment}.
In the moment method, the NDF is projected onto a finite set of moments, yielding PDE systems for macroscopic observables. Beyond their role as computational strategies for kinetic equations, moment methods provide natural frameworks for deriving hydrodynamic theories of fluid-like systems and are thus inherently {\em multiscale}. However, the resulting system always requires \textit{moment closure} to express higher-order moments in terms of the retained ones. Common closures include the Maxwellian closure \cite{sahu2020full}, Grad’s method \cite{grad1949kinetic}, maximum-entropy closures \cite{dreyer1987maximisation, levermore1996moment, ruggeri1993extended}, and quadrature-based moment methods (QBMM) \cite{chen2025poisson, fox2008quadrature, fox2018conditional, fox2009higher, marchisio2005solution, patel2019three}, the latter reconstructing the NDF from a small number of discrete velocities and weights. Owing to their flexibility in representing highly non-equilibrium distributions and capturing multi-modal features with relatively few moments, QBMM have become a powerful and versatile tool, and form the focus of the present work. 

The accuracy and stability of moment-closure systems rely on two vital properties: hyperbolicity, which ensures the equations are well-posed for wave propagation, and moment realizability. A major advantage of QBMM, especially the original quadrature method of moments (QMOM), lies in their intrinsic realizability property: if the given moment vector is realizable, the associated quadrature weights are nonnegative, ensuring that higher-order moments reconstructed from the quadrature remain within the realizable set. Previous works on QBMM therefore concentrated on analyzing the hyperbolicity of the resultant moment closure systems. It was shown that the original QMOM lacks global hyperbolicity for Boltzmann-type kinetic equations \cite{chalons2012beyond,huang2020stability}. Two notable extensions were proposed in 1D to address this issue: Gaussian-EQMOM, which approximates the NDF by the weighted sum of Gaussian distributions sharing a common variance \cite{chalons2017multivariate}, and HyQMOM , which uses more delta functions than the number of transported moments, introducing extra degrees of freedom to parameterize part of the quadrature from low-order moments \cite{fox2018conditional,fox2022hyqmom}. Both approaches restore global hyperbolicity and expand the applicability of QBMM to various problems \cite{TAUNAY2023111700}.

In addition to ensuring realizability at the closure level, it is equally important to preserve this property when solving the resulting moment-closure systems. This requirement is analogous to positivity preservation for the compressible Euler and Navier–Stokes equations, where quantities such as density, pressure, and internal energy must remain positive to maintain physical consistency. Numerical violations of these constraints often lead to instability or even blow-up of the solver \cite{ZS2}. 
Over the past decades, numerous bound-preserving methods have been developed for hyperbolic conservation laws. The seminal work of Zhang and Shu \cite{ZS1,ZS2} introduced a unified convex decomposition approach for constructing positivity-preserving discontinuous Galerkin (DG) and finite volume schemes, which was later extended to the compressible Navier–Stokes equations \cite{Zxx} and to ideal magnetohydrodynamics \cite{wu2018positivity,wu2019provably} by coupling positivity preservation with discrete divergence-free constraints. Cui, Ding, and Wu further established an optimal convex decomposition theory \cite{cui2023classic,cui2024optimal} to improve the CFL constraint and efficiency.
To address the challenges posed by nonlinear constraint preservation, Wu and Shu proposed the geometric quasi-linearization (GQL) framework \cite{wu2023geometric}, which recasts nonlinear constraints as linear ones. Other important advances include bound-preserving exponential Runge–Kutta schemes for stiff systems \cite{huang2018bound}, asymptotic-preserving and positivity-preserving schemes for kinetic equations and multiscale problems \cite{hu2017asymptotic, hu2019second, hu2018asymptotic, jin2022asymptotic}, flux-corrected transport methods \cite{kuzmin2002flux,zalesak1979fully}, parametrized flux limiters \cite{xu2014parametrized}, and convex limiting techniques \cite{guermond2018second,guermond2019invariant}. Comprehensive reviews of high-order bound-preserving methods can be found in \cite{shu2020class,xu2017bound,ZS4}.
 
For QBMM-derived moment closure systems, realizability-preserving schemes have been explored in various contexts. Notably, the pioneering work of `kinetic flux' was proposed in finite volume schemes \cite{desjardins2008quadrature} and was later extended to third-order schemes \cite{fox2008quadrature}. However, only the first-order scheme for the two-node QMOM-derived system was formally proved realizable \cite{desjardins2008quadrature}. Vikas et al. \cite{vikas2011realizable,vikas2013realizable} developed quasi-high-order finite-volume methods by combining high-order weight reconstruction with first-order abscissa reconstruction, achieving realizability but with reduced accuracy compared to standard high-order schemes. As QBMM is also widely used to treat population balance equations for the evolution of particle size distributions, realizability-preserving second-order schemes have been developed for a series of gas-particle/droplet flow problems using canonical moments and the kinetic flux \cite{kah2012high,laurent2017realizable}. Another related advance is the high-order Runge-Kutta DG scheme with a realizability-limiter for solving the one-dimensional moment system derived with the maximum-entropy closure \cite{alldredge2015realizability}, but the method seems not directly extendable to QBMM-derived systems. 
Therefore, it becomes clear that no general high-order finite-volume method with kinetic fluxes exists that can provably preserve moment realizability for (strongly-nonlinear) QBMM closures. While existing strategies are often tailored to specific closures or simplified physics, their extension to complex closures like Gaussian-EQMOM and HyQMOM remains challenging due to strong nonlinear coupling and the lack of explicit primitive–flux relationships.  

This paper addresses the above issue by developing realizability-preserving finite-volume schemes for two moment models associated with the Gaussian-EQMOM and HyQMOM closures. Rather than relying on kinetic fluxes, we adopt a GQL-inspired approach that recasts the realizability condition as a nonnegative quadratic form in the moment vector, thereby reducing the original nonlinear constraints to bilinear inequalities that are more tractable for theoretical analysis and verification. Building on this formulation, we construct a Harten--Lax--van Leer (HLL) numerical flux with rigorously estimated wave speeds and intermediate states, embedding realizability directly into the flux evaluation. The approach applies to both Gaussian-EQMOM and HyQMOM; the only change is the set of quadratic-form coefficients tailored to their respective moment structures. In this way, we obtain a unified framework for realizability preservation across QBMM variants. A rigorous analysis provides sufficient realizability-preserving conditions under Courant--Friedrichs--Lewy (CFL) constraints, and a tailored limiter ensures the physical admissibility of reconstructed polynomials in practical computations. Numerical experiments confirm the method’s robustness and accuracy for challenging nonlinear moment-closure problems.
 
The rest of this paper is organized as follows. Section \ref{sec2} presents the moment closure approaches, key observations and lemmas for the realizability-preserving analysis. Section \ref{sec3} details the finite-volume schemes. Section \ref{sec4} reports numerical experiments demonstrating accuracy and robustness. Conclusions are given in Section \ref{sec5}.

\section{Preliminary}\label{sec2}

This section is divided into two parts: Section \ref{sec2.1} introduces quadrature-based moment models of kinetic equations and the moment-closure problem, Section \ref{sec2.2} discusses the realizability of quadrature-based moment models.

\subsection{Quadrature-based moment models of kinetic equations} \label{sec2.1}
Consider the hypothetical one-dimensional kinetic equation without external force terms, given by
\begin{equation}\label{eq_BGK1d}
    \partial _t f+v\partial _x f =\mathcal{Q}(f), 
\end{equation}
where the unknown $f=f(t,x,v)$ denotes the number density function (NDF), defined on the domain $(t,x,v)\in \mathbb{R}_{+}\times \mathbb R\times \mathbb R$.
The collision operator $\mathcal{Q}(f)$ is modeled using the Bhatnager--Gross--Krook (BGK) approximation \cite{BGK1954}, $\mathcal{Q}(f)=\frac{1}{\tau}(f^{eq}-f)$, where $\tau$ is the relaxation time and $f^{eq}$ is the local Maxwellian distribution defined as $f^{eq}(v;\rho,U,\theta)= \frac{\rho}{\sqrt{2\pi\theta}}\exp\left(-\frac{(v-U)^2}{2\theta}\right)$, 
Here $\rho(t,x)=\int_{\mathbb{R}}f{\rm d}v$, $U(t,x)=\int_{\mathbb{R}}vf{\rm d}v$, and $\theta(t,x)=\int_{\mathbb{R}}v^2f{\rm d}v$ representing the macroscopic density, velocity, and temperature of the gas, respectively. 
Define the $k$th velocity moment of NDF $f$ as $M_k=M_k(t,x) = \int_\mathbb R v^k f {\rm d}v$. Multiplying (\ref{eq_BGK1d}) by $v^{k}$ and integrating over $\mathbb{R}$ yield 
\begin{equation} \label{eq_intform}
	\partial_{t}M_k+\partial_{x}M_{k+1}=S_k, \quad k=0,1,2,\cdots,  
\end{equation} 
where $S_k:=\int_{\mathbb{R}}v^k \mathcal{Q}(f) {\rm d}v$ is the source term. Evidently, computing the $k$th-order moment requires knowledge of the $(k+1)$th-order moment, leading to an infinite hierarchy of coupled equations. Truncating this hierarchy at order $n\ge 2$ yields a finite-dimensional system governing the moment vector $\mathbf{M} =(M_0, M_1,\ldots,M_{n})^{\top}$:
\begin{equation}\label{eq_mom_trunc}
    \partial_t \mathbf{M} + \partial_x \mathbf{F}(\mathbf{M}) =\mathbf{S},
\end{equation}
where $\mathbf{F}(\mathbf{M}):= (M_1,\ldots,M_{n},\overline{M}_{n+1})^{\top}$ and $\mathbf{S}=\mathbf S(\mathbf M):=(S_0, S_1,\ldots,S_{n})^{\top}$. Here the notation $\overline M_{n+1}$ indicates that the $(n+1)$th moment is not included in the resolved moment vector $\mathbf M$, and \textit{the moment-closure problem} consists in prescribing a mapping $\overline M_{n+1} = \overline M_{n+1}(\mathbf M)$, 
so that the system (\ref{eq_mom_trunc}) becomes closed. 

\begin{remark}
    With the BGK collision $\mathcal{Q}(f)=\frac{1}{\tau}(f^{eq}-f)$, the source term in (\ref{eq_mom_trunc}) is
    \begin{equation} \label{eq_bgk_mom_source}
        S_k = \frac{1}{\tau}(\rho\Delta_k(U,\theta) - M_k)
    \end{equation}
    for $k=0,\ldots,n$, where $\Delta_k(U,\theta):=\int v^k f^{eq}(v;1,U,\theta)dv$ is the $k$th moment of the Maxwellian $f^{eq}$ with unit mass $\rho=1$. Since $\Delta_k(U,\theta)$ depends only on $M_0$, $M_1$, and $M_2$, it follows that for $n\ge 2$, the source term $\mathbf S$ in (\ref{eq_mom_trunc}) is a \textit{closed} function of $\mathbf M$. In contrast, for other collision models, the source term may also require a closure relation.
\end{remark}

Moment closure is usually achieved by reconstructing a (parametrized) NDF $f^*$ that is consistent with the resolved moment vector $\mathbf{M}$; unclosed quantities such as $\overline M_{n+1}$ are then computed using this distribution. Here we mainly condiser two types of quadrature-based moment models. The first is the Gaussian extended-QMOM (Gaussian-EQMOM) with the ansatz for the reconstructed NDF \cite{chalons2017multivariate}: $ f^{*}(t,x,v) = \sum^{N}_{i=1} \frac{\rho_i}{\sqrt{2\pi}\sigma}\exp\left(-\frac{(v-v_i)^2}{2\sigma^2}\right)$. Here $(2N+1)$ parameters are included: the weights $\rho_i = \rho_i(t,x)$, the abscissas $v_i = v_i(t,x)$, and a common standard deviation $\sigma=\sigma(t,x)$ for all Gaussian modes. These parameters $(\rho_i,v_i,\sigma)$ are determined by matching the $(2N+1)$ lower moments of the distribution:  
$M_k = \sum_{i=1}^N \rho_i \Delta_k(v_i,\sigma^2)$ for $k=0,\ldots,2N,$
where $\Delta_k(v_i,\sigma^2)$ defined in \eqref{eq_bgk_mom_source} denotes the $k$th moment of a Gaussian distribution. Since they are highly nonlinear systems, they can only be computed using specialized moment-inversion algorithms. Fortunately, efficient algorithms exist at least for small $N$ \cite{PIGOU2018243}. We mainly focus on the case $N=2$ for the system \eqref{eq_mom_trunc}.
 
When $N=2$ in \eqref{eq_mom_trunc}, the moment vector  ${\mathbf{M}}=(M_0,M_1,M_2,M_3,M_4)^\top$ (termed as the ``conservative variables''), the flux $\mathbf{F}({\mathbf{M}})= (M_1,M_2,M_3,M_4,\overline{M}_5)^{\top}$ and  the source term $\mathbf{S}= (S_0,S_1,S_2,S_3,S_4)^{\top}$ with $S_0=S_1=S_2=0$ and $S_k = \frac{1}{\tau}(\rho\Delta_k(U,\theta) - M_k)$ for $k=3,4$, where
$ \rho = M_0, U=\frac{M_1}{M_0}, \theta=\frac{M_2}{M_0}-U^2, \Delta_3(U,\theta) = U^3+3U\theta, \Delta_4(U,\theta) = U^4+6U^2\theta+3\theta^2$. 
When the system \eqref{eq_mom_trunc} is closed by Gaussian-EQMOM, the flux is 
\begin{equation}\label{eq_eqm2node_M5}
	\overline{M}_5=\rho_1 v^5_1 +\rho_2 v^5_2 +10M_3\sigma^2-15\sigma^2M_1,
\end{equation}
where the primitive variables, collected as $\mathbf{W}=(\rho_1,v_1,\rho_2,v_2,\sigma)$, are obtained by solving the moment-inversion system
\begin{equation}\label{eq_InvM_Gaussian}
	\begin{split}
		M_0=&\rho_1+\rho_2,\\
		M_1=&\rho_1 v_1 + \rho_2 v_2,\\
		M_2=&\rho_1 v^2_1 + \rho_2 v^2_2+\sigma^2M_0,\\
		M_3=&\rho_1 v^3_1 + \rho_2 v^3_2+3\sigma^2M_1,\\
		M_4=&\rho_1 v^4_1 + \rho_2 v^4_2+6\sigma^2M_2 - 3\sigma^4M_0.
	\end{split}
\end{equation}
For moment inversion of \eqref{eq_InvM_Gaussian}, one can solve $\sigma$ first from a third-order polynomial and then $(\rho_i,v_i)$ from the first four equations \cite{chalons2017multivariate}.

The next model we consider is the hyperbolic-QMOM (HyQMOM) with $N=2$, which is built based on a ``three-point'' ansatz $f^*(t,x,v)=\sum_{i=1}^N \rho_i\delta(v-v_i)$ \cite{fox2018conditional}. In this case, we still trace the moment vector $\bfM = (M_0,M_1,M_2,M_3,M_4)^\top$ by the system \eqref{eq_mom_trunc}, and the flux $\bfF(\bfM)$ now contains the following closure relation 
\begin{equation}\label{eq_hyq2node_M5}
	\overline{M}_5=\rho_1 v^5_1 +\rho_2 v^5_2+\rho_3 v^5_3.
\end{equation}
The primitive variables $\mathbf{W}=(\rho_1,\rho_2,\rho_3,v_1,v_3)$ are determined from (noticing that $v_2$ is given by the lower-order moments $M_0$ and $M_1$) 
\begin{equation}\label{eq_InvM_HyQMOM} 
	\begin{aligned}
		M_0&=\rho_1       + \rho_2       + \rho_3,\\
		M_1&=\rho_1 v_1   + \rho_2 v_2   + \rho_3 v_3,\\
		M_2&=\rho_1 v^2_1 + \rho_2 v_2^2 + \rho_3 v^2_3,\\
		M_3&=\rho_1 v^3_1 + \rho_2 v_2^3 + \rho_3 v^3_3,\\
		M_4&=\rho_1 v^4_1 + \rho_2 v_2^4 + \rho_3 v^4_3,\\
		v_2&= U = {M_1}/{M_0}.
	\end{aligned}
\end{equation}
The system (\ref{eq_InvM_HyQMOM}) admits an analytical solution for $\mathbf{W} = \mathbf{W}(\mathbf M)$ \cite{johnson2023positivity}, facilitating highly efficient moment inversion. More discussions on HyQMOM can be found in \cite{fox2022hyqmom, fox2018conditional, johnson2023positivity, zhang2024hyqmom}.

\subsection{Realizability of quadrature-based moment models}\label{sec2.2}

This subsection considers the central question: realizability. A moment vector $\mathbf M$ is called \textit{realizable} if there exists a non-negative distribution $f^{*}(v)\ge0$ such that $M_k=\int_{\mathbb{R}}v^kf^{*}(v){\rm d}v$ for $k=0,\ldots,n$. This realizability constitutes a necessary condition for $\mathbf M$ to originate from a physically meaningful NDF. The collection of all realizable $\mathbf M$ forms a subset of $\mathbb R^{n+1}$, known as the \textit{moment space} (or the {\em admissible moment set}). It is a proper subset of $\mathbb R^{n+1}$ because, for instance, one always has $M_{2k}\ge 0$ for $k=0,1,\ldots$. Moreover, $\mathbf M$ is called \textit{strictly realizable} if it lies in the \textit{interior} of the moment space.

Indeed, the moment space, particularly its interior, can be fully characterized using the Hankel matrix associated with $\mathbf M$. To state the results, we now assume $n=2N$ for $N\in\mathbb N$, i.e., $\mathbf{M} =(M_0, M_1,\ldots,M_{2N})^{\top}$. This suffices for our applications, and the case $n=2N-1$ can also be treated (see \cite{shohat1943problem}). The Hankel matrix is defined as
\begin{equation}
    \mathcal H_{N} = \mathcal H_{N}(\mathbf M) =
    \begin{bmatrix}
        M_0 & M_1 & \cdots & M_N \\
        M_1 & M_2 & \cdots & M_{N+1} \\
        \vdots & \vdots &  & \vdots \\
        M_N & M_{N+1} & \cdots & M_{2N}
    \end{bmatrix}
    \in \mathbb R^{(N+1)\times (N+1)}.
\end{equation}
Denote by $\Omega_N$ the interior of the moment space. We have the following
\begin{theorem}[Theorem 1.2 of \cite{shohat1943problem}] \label{Th1_Hankel}
    A moment vector $\mathbf M \in \Omega_N$ if and only if the Hankel matrix $\mathcal H_{N}(\mathbf M)$ is positive definite (i.e., strictly realizable). Instead, $\mathbf M$ lies on the boundary of the moment space if and only if $\mathcal H_N(\mathbf M)$ is positive semi-definite with $\det \mathcal H_N(\mathbf M) = 0$ (i.e., realizable). In the latter case, $\mathbf M$ is generated by a weighted sum of $r$ Dirac delta functions, where $r\le N$ is the rank of $\mathcal H_N(\mathbf M)$. 
\end{theorem}

Therefore, $\Omega_N$ is the set of moment vectors for which the associated Hankel matrix is positive definite. It follows that $\Omega_N$ is an open convex set.
For the case of $N=2$, namely, $\mathbf M=(M_0,M_1,M_2,M_3,M_4)^\top\in\mathbb R^5$. Then, by the above theorem, $\Omega_2$ can be explicitly characterized as
\begin{equation}\label{eq2:moment_realizability}
    \Omega_2 =\left\{ \mathbf{M}\in\mathbb R^5: M_0>0, \ \mathcal{E}(\mathbf{M})=e>0, \  \mathcal{Z}(\mathbf{M})=\eta-(e^2+\frac{q^2}{e})>0\right\}, 
\end{equation}
where $e:={(M_0M_2-M^2_1)}/{M^2_0}$, $q:={((M_3M^2_0-M^3_1)-3M_1(M_0M_2-M^2_1))}/{M^3_0}$, and  $\eta:={(-3M^4_1+M_4M^3_0-4M^2_0M_1M_3+6M_0M_1^2M_2)}/{M^4_0}$. 
The inequalities in \eqref{eq2:moment_realizability} mean that all leading principal minors of $\mathcal H_2(\mathbf M)$ are positive. 
For Gaussian-EQMOM, it is more natural to characterize realizability in terms of $\mathbf W$, leading to
$\Omega_{\mathbf W}:=\{\mathbf W\in\mathbb R^5:\ \rho_1>0, \ \rho_2>0, \ \sigma>0\}$.
The existence of $\mathbf W\in\Omega_{\mathbf W}$ is ensured for almost every $\mathbf M\in \Omega_2$ (the interior of the moment space in $\mathbb R^5$, \eqref{eq2:moment_realizability}); indeed, a zero-measured lower-dimensional subset $\{\mathbf M: \ q=0,\ \eta>3e^2\}\subset\mathbb R^5$ has to be excluded \cite{chalons2017multivariate}.  
For HyQMOM, the realizability naturally introduces $ \Omega_{\mathbf W}:=\{\mathbf W\in\mathbb R^5:\ \rho_1>0, \ \rho_2>0, \ \rho_3>0\}$. The existence of $\mathbf W \in \Omega_{\mathbf W}$ is ensured for every realizable $\mathbf M$.  

Since both Gaussian-EQMOM and HyQMOM yield hyperbolic systems \cite{chalons2017multivariate,fox2018conditional,huang2020stability}, it is crucial to ensure the realizability of the transported moment vector $\mathbf M$ during numerical simulations. Violating the condition may result in loss of hyperbolicity, nonphysical oscillations, or even blow-up of the scheme. These issues will be addressed in the next section.

\section{Realizability-preserving finite volume schemes}\label{sec3}

This section develops provably realizability-preserving finite volume schemes for the five-moment systems \eqref{eq_mom_trunc} obtained with the Gaussian-EQMOM \eqref{eq_eqm2node_M5} and HyQMOM \eqref{eq_hyq2node_M5}  in the case $N=2$. With key preparation steps presented in Section \ref{sec3.1}, the schemes are developed in Sections \ref{sec3.2} and \ref{sec3.3}.

\subsection{Auxiliary results}\label{sec3.1}

This subsection derives several fundamental Lemmas for the construction of provably realizability-preserving numerical schemes.

\begin{lemma}\label{prop1}
Let $\mathbf{M} \in \Omega_2$ and let $\mathbf F = \mathbf F(\mathbf M)$ be defined by \eqref{eq_eqm2node_M5} and \eqref{eq_InvM_Gaussian} for the two-node Gaussian-EQMOM. For any constant $a>\sqrt{3}$, define
\begin{subequations} \label{eq:delta_mp}
\begin{align} 
    \delta^+ = \delta^+(\bfM) = \max \{ v_1+a\sigma, \ v_2+a\sigma, \ 0 \}, \\
    \delta^- = \delta^-(\bfM) = \min \{ v_1-a\sigma, \ v_2-a\sigma, \ 0 \}.
\end{align}
\end{subequations}
Then $\delta^{+}\mathbf{M} -\mathbf{F}$ and $\mathbf{F}-\delta^{-}\mathbf{M}$ both belong to $\Omega_2$.
\end{lemma}

\begin{proof}
    By Theorem \ref{Th1_Hankel}, it suffices to verify that the Hankel matrices $\mathcal H_2(\delta^{+}\mathbf{M} -\mathbf{F})$ and $\mathcal H_2(\mathbf{F}-\delta^{-}\mathbf{M})$ are positive definite. Recall that
    \[
	\mathcal H_2(\mathbf M) =
        \begin{bmatrix}
            M_0 & M_1 & M_2\\ M_1 & M_2 & M_3	\\ M_2 & M_3 & M_4
        \end{bmatrix},
        \quad
        \mathcal H_2(\mathbf F) =
        \begin{bmatrix}
            M_1 & M_2 & M_3	\\ M_2 & M_3 & M_4 \\M_3 & M_4 & M_5
        \end{bmatrix}.
    \]
    From \eqref{eq_InvM_Gaussian}, these matrices can be decomposed as
    \[
        \mathcal H_2(\mathbf M) = \rho_1 A_1 + \rho_2 A_2, \quad
        \mathcal H_2(\mathbf F) = \rho_1 B_1 + \rho_2 B_2,
    \]
    where
    \[
        A_i = 
        \begin{bmatrix}
            \Delta_0(v_i,\sigma^2) & \Delta_1(v_i,\sigma^2) & \Delta_2(v_i,\sigma^2) \\
            \Delta_1(v_i,\sigma^2) & \Delta_2(v_i,\sigma^2) & \Delta_3(v_i,\sigma^2) \\
            \Delta_2(v_i,\sigma^2) & \Delta_3(v_i,\sigma^2) & \Delta_4(v_i,\sigma^2)
        \end{bmatrix}, \
        B_i = 
        \begin{bmatrix}
            \Delta_1(v_i,\sigma^2) & \Delta_2(v_i,\sigma^2) & \Delta_3(v_i,\sigma^2) \\
            \Delta_2(v_i,\sigma^2) & \Delta_3(v_i,\sigma^2) & \Delta_4(v_i,\sigma^2) \\
            \Delta_3(v_i,\sigma^2) & \Delta_4(v_i,\sigma^2) & \Delta_5(v_i,\sigma^2) \\
        \end{bmatrix}
    \]
for $i=1,2$, with $\Delta_k(v_i,\sigma^2)$ denoting the $k$-th moment of a Gaussian distribution as defined in \eqref{eq_bgk_mom_source}. Hence
    \[
        \mathcal H_2(\delta^{+}\mathbf{M} -\mathbf{F}) = \sum_{i=1}^2 \rho_i (\delta^+ A_i - B_i), \quad
        \mathcal H_2(\mathbf{F}-\delta^{-}\mathbf{M}) = \sum_{i=1}^2 \rho_i (B_i - \delta^- A_i).
    \]
    Thus, it suffices to show that $\delta^+ A_i - B_i$ and $B_i - \delta^- A_i$ are positive definite for $i=1,2$. A direct calculation gives
    \[
        (v_i+a\sigma)A_i - B_i = \sigma\mathcal H_2(\bm{\beta}), \quad
        B_i - (v_i-a\sigma)A_i = \sigma\mathcal H_2(\bm{\gamma}),
    \]
    where $\bm{\beta} = (\beta_0,\beta_1,\beta_2,\beta_3,\beta_4)^\top$ and $\bm{\gamma} = (\gamma_0,\gamma_1,\gamma_2,\gamma_3,\gamma_4)^\top$ with  
    \[
        \beta_k = a\Delta_k(v_i,\sigma^2) - k\sigma\Delta_{k-1}(v_i,\sigma^2), \quad
        \gamma_k = a\Delta_k(v_i,\sigma^2) + k\sigma\Delta_{k-1}(v_i,\sigma^2)
    \] 
    for $k=0,\ldots,4$.
    Here, we use the recurrence relation \cite{huang2020stability} for Gaussian moments $\Delta_k(v,\sigma^2)$: 
    $$\Delta_{k+1}(v,\sigma^2) = v\Delta_k(v,\sigma^2)+k\sigma^2\Delta_{k-1}(v,\sigma^2) \mbox{ with } \Delta_{-1}(v_i,\sigma^2)=0 \mbox{ and } \Delta_0(v,\sigma^2)=1.$$ 

    Let $P_k$ denote the $k$-th leading principal minor of both $\mathcal H_2(\bm{\beta})$ and $\mathcal H_2(\bm{\gamma})$ (these minors coincide for both matrices). The positive definiteness of these matrices requires
    \begin{equation} \label{eq_leading_principal_minor}
        P_1 = a>0, \quad P_2=\sigma^2(a^2-1)>0, \quad 
        P_3=2a\sigma^6(a^2-3)>0,
    \end{equation}
    which holds if and only if $a>\sqrt{3}$.

    Since $A_i$ is positive definite as the Hankel matrix associated with Gaussian moments $(\Delta_0(v_i,\sigma^2),\ldots,\Delta_4(v_i,\sigma^2))^\top$, for any nonzero vector $x\in\mathbb R^3$, it follows that
    \[ 
    \begin{aligned}
        x^\top (\delta^+ A_i - B_i) x \ge x^\top ((v_i+a\sigma)A_i-B_i)x>0, \\
        x^\top (B_i - \delta^- A_i) x \ge x^\top (B_i - (v_i-a\sigma)A_i)x>0.
    \end{aligned}
    \]
    This completes the proof.
\end{proof}

\begin{remark}
	If $a=\sqrt{3}$, then $\delta^{+}\mathbf{M} -\mathbf{F}$ and $\mathbf{F}-\delta^{-}\mathbf{M}$ remain realizable and, in most cases, lie in the interior of $\Omega_2$. Specifically, if $v_1\ne v_2$, then $\delta^+ > v_i + \sqrt{3}\sigma$ for some $i$, ensuring the positive definiteness of $\mathcal H_2(\delta^{+}\mathbf{M} -\mathbf{F})$, and similarly for $\mathcal H_2(\mathbf{F}-\delta^{-}\mathbf{M})$. In the degenerate case $v_1=v_2=:v$, where $\mathbf M$ corresponds to a single Gaussian mode lying on a measure-zero lower-dimensional manifold of $\Omega_2\subset\mathbb R^5$, we have $A_i=:A$, $B_i=:B$, and $\delta^+=v+\sqrt{3}\sigma$. Examining the principal minors shows that $P_3$ in \eqref{eq_leading_principal_minor} vanishes, so $\delta^{+}\mathbf{M} -\mathbf{F}$ lies on the boundary of the admissible moment set $\Omega_2$ and can be represented as a weighted sum of two Dirac delta functions by Theorem~\ref{Th1_Hankel}. An analogous statement holds for $\mathbf{F}-\delta^{-}\mathbf{M}$.
\end{remark}

Similarly, we have the following results for HyQMOM.

\begin{lemma}\label{prop2}
    Let $\mathbf{M} \in \mathbb R^5$ be realizable, and let $\mathbf F = \mathbf F(\mathbf M)$ be defined by \eqref{eq_hyq2node_M5} and \eqref{eq_InvM_HyQMOM}. Define
    \begin{subequations} \label{eq:delta_mp_hyq}
    \begin{align} 
        &\delta^+ = \delta^+(\bfM) = \max \{ v_1, \ v_2, \ v_3, \ 0 \}, \\
        &\delta^- = \delta^-(\bfM) = \min \{ v_1, \ v_2, \ v_3, \ 0 \}.
    \end{align}
    \end{subequations}
    Then both $\delta^{+}\mathbf{M} -\mathbf{F}$ and $\mathbf{F}-\delta^{-}\mathbf{M}$ are realizable and   
     both lie in the closure of the admissible moment set $\Omega_2$. 
\end{lemma}

\begin{proof}
    By arguments analogous to those in the proof of Lemma~\ref{prop1},
    \[
        \mathcal H_2(\delta^{+}\mathbf{M} -\mathbf{F}) = \sum_{i=1}^3 \rho_i (\delta^+ A_i - B_i), \quad
        \mathcal H_2(\mathbf{F}-\delta^{-}\mathbf{M}) = \sum_{i=1}^3 \rho_i (B_i - \delta^- A_i),
    \]
    where
    \[
        A_i = 
        \begin{bmatrix}
	   1    & v_i   &  v_i^2\\
	   v_i  & v_i^2 &  v_i^3\\
	   v_i^2& v_i^3 &  v_i^4
	\end{bmatrix}, \quad
	B_i = 
        \begin{bmatrix}
	   v_i   & v_i^2 &  v_i^3	\\
	   v_i^2 & v_i^3 &  v_i^4 \\
	   v_i^3 & v_i^4 &  v_i^5
	\end{bmatrix}.
    \]
    Note that $v_iA_i-B_i = 0$ and each $A_i$ is positive semidefinite. Hence $x^\top (\delta^+ A_i - B_i) x \ge 0$ and $x^\top (B_i - \delta^- A_i) x \ge 0$ for all $x\in\mathbb R^3$, which proves the claim.
\end{proof}

The next result is useful for analyzing the effect of the source term in \eqref{eq_mom_trunc} on realizability preservation.
\begin{lemma}\label{pro3}
Let $\mathbf M\in \Omega_2$ and let the source be $\mathbf S(\bfM)=(0,0,0,S_3,S_4)^\top$. Then $\mathbf M+\mu \mathbf S$ is realizable for $\mu \in [0,\mu_0]$ and lies in $\Omega_2$ for $\mu\in[0,\mu_0)$. Here, $\mu_0$ (dependent on $\mathbf M$ and $\tau$) is the unique positive root of the quadratic $\mathcal{P}(\mu)=a_0\mu^2+a_1\mu+a_2$, with coefficients
\begin{equation}\label{eq:abc}
\begin{aligned}
    a_0&=-M_0S^2_3<0,\\
    a_1&=2S_3(M_1M_2-M_0M_3)+S_4(M_0M_2-M^2_1),\\
    a_2&=\det(\mathcal H_2(\mathbf M))=M_0M_2M_4-M_0M^2_3-M^2_1M_4+2M_1M_2M_3-M^3_2>0.
\end{aligned}
\end{equation} 
\end{lemma}

\begin{proof} 
    Using the notation in \eqref{eq2:moment_realizability}, a straightforward calculation yields 
    \begin{equation}\label{eq1}
        \mathcal{E}(\mathbf{M}+\mu\mathbf{S})=\frac{M_0M_2-M^2_1}{M^2_0} = \mathcal E(\mathbf M)>0,\quad
	\mathcal{Z}(\mathbf{M}+\mu\mathbf{S})=\frac{\mathcal{P}(\mu)}{M_0(M_0M_2-M^2_1)},
    \end{equation}
   where the quadratic $\mathcal P(\mu)$ has $a_0<0$ and $a_2>0$ and therefore a unique positive root $\mu_0$ such that $\mathcal P(\mu)>0$ for $\mu\in[0,\mu_0)$. This proves the result.
\end{proof}

\subsection{Realizability-preserving scheme in collisionless case}\label{sec3.2}
We now construct realizability-preserving finite-volume schemes for the five-moment systems \eqref{eq_mom_trunc}. We first consider the collisionless case  $\partial_t \mathbf M + \partial_x \mathbf F(\mathbf M) = 0$ and then incorporate the source term in Section \ref{sec3.3}. On a uniform grid with forward Euler time stepping (higher-order time discretization is discussed in Remark~\ref{rem:high_order}), the scheme reads
\begin{equation}\label{1d_HLL_2ndScheme}
	\begin{aligned}
		\overline{\bfM}^{n+1}_{i}&=\overline{\bfM}^{n}_{i} -\frac{\Delta t_n}{\Delta x}\left(\widehat{\bfF}( {\bfM}^{n,-}_{i+\frac12}, {\bfM}^{n,+}_{i+\frac12}) - \widehat{\bfF}({\bfM}^{n,-}_{i-\frac12}, {\bfM}^{n,+}_{i-\frac12}) \right), 
	\end{aligned}
\end{equation}
where ${\overline{\bfM}^n_i=(\overline{M_0}^{n}_{i},\ldots,\overline{M_4}^{n}_{i})^{\top}}$ approximates the cell average of $\mathbf{M}(t^n,x)$ on $[x_{i-\frac12},x_{i+\frac12}]$, and ${\widehat{\bfF}( {\bfM}^{n,-}_{i+\frac12}, {\bfM}^{n,+}_{i+\frac12})}$ is the numerical flux to be specified.

Our goal is to ensure that, if the cell averages $\overline \bfM^n_i$ are realizable for all $i$, then the updates $\overline \bfM^{n+1}_i$ remain realizable under suitable CFL conditions. Direct verification is difficult because the realizability domain (e.g., $\Omega_2$ in \eqref{eq2:moment_realizability}) involves nonlinear constraints, and the right-hand side of \eqref{1d_HLL_2ndScheme} is nonlinear in $\overline \bfM_i^n$ through the numerical flux. Inspired by the GQL idea \cite{wu2023geometric}, we convert the nonlinear constraints on $\bfF(\bfM)$ into linear constraints on $\bfM$.

In this work, we employ the HLL flux with two wave speeds $\delta^L<\delta^R$:
\begin{equation}\label{eq:HLLflux_1}
	\begin{aligned}
		\widehat{\bfF}(\bfM^-, \bfM^+)&=
		\begin{cases}
                \mathbf{F}(\bfM^-),&0\le\delta^L<\delta^R,\\
                \frac{\delta^R \bfF(\bfM^-) - \delta^L \bfF( \bfM^+) + \delta^R\delta^L(\bfM^+ - \bfM^-)}{\delta^R-\delta^L},&\delta^L<0<\delta^R,\\
			\mathbf{F}(\bfM^+),&\delta^{L}<\delta^{R}\le0,
		\end{cases} 
	\end{aligned}
\end{equation}
which can be rewritten by setting $\hat\delta^+=\max\{\delta^R,0\}$ and $\hat\delta^-=\min\{\delta^L,0\}$ as
\begin{equation}\label{eq:HLLflux_2}
    \widehat{\bfF}(\bfM^-,\bfM^+) = \frac{\hat\delta^+ \bfF(\bfM^-) - \hat\delta^- \bfF(\bfM^+) + \hat\delta^+\hat\delta^- (\bfM^+ - \bfM^-)}{\hat\delta^+ -\hat\delta^-}.
\end{equation}
For realizability-preserving considerations, we specify $\hat\delta^{\pm}$ as 
\begin{subequations}
\begin{align}
    \hat\delta^+ &= \hat\delta^+(\bfM^+,\bfM^-) = \max\{\delta^+(\bfM^+), \delta^+(\bfM^-)\}, \\
    \hat\delta^- &= \hat\delta^-(\bfM^+,\bfM^-) = \min\{\delta^-(\bfM^+), \delta^-(\bfM^-)\},
\end{align}
\end{subequations}
where $\delta^{\pm}(\bfM)$ are given by \eqref{eq:delta_mp} for the two-node Gaussian-EQMOM (depending on a constant $a>\sqrt{3}$) and by \eqref{eq:delta_mp_hyq} for the three-point HyQMOM.

To achieve second-order spatial accuracy, we reconstruct the interface values $\{\bfM^{n,\mp}_{i\pm\frac12}\}$ from the given cell averages $\{\overline{\bfM}^{n}_i\}$ using piecewise linear functions $\bfM(t^n,x)=\overline{\bfM}^{n}_i + (\widehat{\bfM}^{n}_x)_i(x-x_i)$ for $x\in [x_{i-\frac12},x_{i+\frac12}]$, yielding
\begin{equation}\label{linear_M}
    \bfM^{n,\mp}_{i\pm\frac12}=\overline{\bfM}^n_i \pm \frac{\Delta x}{2}(\widehat{\bfM}^n_x)_{i}.
\end{equation}
The local slopes $(\widehat{\bfM}_x)_i$ are computed componentwise via the van Albada limiter \cite{van1982comparative}:
\begin{equation}\label{eq:Albada}
    (\widehat{\bfM}^{n}_x)_i=\frac{\left(\left[\frac{\overline{\bfM}_{i+1}-\overline{\bfM}_{i}}{\Delta x}\right]^2+\varepsilon_R\right) \circ \frac{\overline{\bfM}_{i}-\overline{\bfM}_{i-1}}{\Delta x} +\left(\left[\frac{\overline{\bfM}_{i}-\overline{\bfM}_{i-1}}{\Delta x}\right]^2+\varepsilon_L\right)\circ \frac{\overline{\bfM}_{i+1}-\overline{\bfM}_{i}}{\Delta x} }{\left[\frac{\overline{\bfM}_{i}-\overline{\bfM}_{i-1}}{\Delta x}\right]^2+\left[\frac{\overline{\bfM}_{i+1}-\overline{\bfM}_{i}}{\Delta x}\right]^2+\varepsilon_L+\varepsilon_R}.
\end{equation}
Here $[\bfM]^2=(M_0^2,\ldots,M_4^2)^\top$ for any vector $\bfM=(M_0,\ldots,M_4)^\top$, ``$\circ$'' denotes the Hadamard (elementwise) product, and we take $\varepsilon_L=\varepsilon_R=3\Delta x$.

Under an additional assumption on the reconstructed values $\bfM^{n,\mp}_{i\pm\frac12}$, the above scheme \eqref{1d_HLL_2ndScheme} can be proven to preserve  moment realizability. Denote the approximate interface wave speeds as 
\begin{equation} \label{eq:delta_tn_pm_1}
    \delta_{i+\frac12}^{n,\pm} := \hat \delta^\pm (\bfM_{i+\frac12}^{n,+},\bfM_{i+\frac12}^{n,-}),
\end{equation}
which satisfy $\delta_{i+\frac12}^{n,-} \le 0 \le \delta_{i+\frac12}^{n,+}$ and $\delta_{i+\frac12}^{n,-} < \delta_{i+\frac12}^{n,+}$. Then we have the following theorem. 

\begin{theorem}\label{eq:th1}
    Assume that $\overline{\bfM}^{n}_i$ is strictly realizable (i.e., $\overline{\bfM}_i^n\in\Omega_2$) for all $i$ in \eqref{1d_HLL_2ndScheme}. If the reconstructed values $\bfM^{n,\pm}_{i+\frac12}$ are strictly realizable (i.e., $\bfM^{n,\pm}_{i+\frac12}\in\Omega_2$) for all $i$, then the updated $\overline{\bfM}^{n+1}_i$ is strictly realizable for Gaussian-EQMOM and realizable for HyQMOM under the CFL condition
	\begin{equation}\label{eq:cfl1}
		\frac{\Delta t_n}{\Delta x}\max_i\left\{\delta^{n,+}_{i+\frac12} - \delta^{n,-}_{i+\frac12}\right\} 
		\le \frac12,\quad \forall i.
	\end{equation}
\end{theorem}

\begin{proof}
    From \eqref{linear_M} we have 
    \begin{equation}\label{eq:Mn}
        \overline{\bfM}^{n}_{i}=\frac{1}{2}(\mathbf{M}^{n,+}_{i-\frac12}+\mathbf{M}^{n,-}_{i+\frac12}).
    \end{equation}
    Substituting \eqref{eq:HLLflux_2} and \eqref{eq:Mn} into \eqref{1d_HLL_2ndScheme}, we rewrite the scheme \eqref{1d_HLL_2ndScheme} as
    \begin{equation}\label{eq:M_Pi}
	\overline{\bfM}^{n+1}_{i}= \Pi_1+\Pi_2+\Pi_3+\Pi_4,
    \end{equation}  
    where 
    \begin{equation}\label{eq:def_Pi}
    \begin{aligned} 
        \Pi_1 & = \left(\frac12-\frac{\Delta t_n}{\Delta x}\delta^{n,+}_{i-\frac12}\right) \mathbf{M}^{n,+}_{i-\frac12}+\frac{\Delta t_n}{\Delta x}\mathbf{F}(\mathbf{M}^{n,+}_{i-\frac12}),  
	\\
        \Pi_2 & = \left(\frac12+\frac{\Delta t_n}{\Delta x}\delta^{n,-}_{i+\frac12}\right) \mathbf{M}^{n,-}_{i+\frac12}-\frac{\Delta t_n}{\Delta x}\mathbf{F}(\mathbf{M}^{n,-}_{i+\frac12}),  
	\\
        \Pi_3 & =\frac{\frac{\Delta t_n}{\Delta x}\delta^{n,+}_{i-\frac12}}{\delta^{n,+}_{i-\frac12}-\delta^{n,-}_{i-\frac12}}\Big[ \delta^{n,+}_{i-\frac12}\mathbf{M}^{n,+}_{i-\frac12}-\mathbf{F}(\mathbf{M}^{n,+}_{i-\frac12})-\delta^{n,-}_{i-\frac12}\mathbf{M}^{n,-}_{i-\frac12}+\mathbf{F}(\mathbf{M}^{n,-}_{i-\frac12})\Big],
	\\
        \Pi_4 & =\frac{-\frac{\Delta t_n}{\Delta x}\delta^{n,-}_{i+\frac12}}{\delta^{n,+}_{i+\frac12}-\delta^{n,-}_{i+\frac12}}\Big[ \delta^{n,+}_{i+\frac12}\mathbf{M}^{n,+}_{i+\frac12}-\mathbf{F}(\mathbf{M}^{n,+}_{i+\frac12})-\delta^{n,-}_{i+\frac12}\mathbf{M}^{n,-}_{i+\frac12}+\mathbf{F}(\mathbf{M}^{n,-}_{i+\frac12})\Big]. 
    \end{aligned}
    \end{equation} 

    We now proceed to show that $\mathcal H_2(\Pi_k)>0$ ($k=1,\ldots,4$) for Gaussian-EQMOM (resp.~$\mathcal H_2(\Pi_k)\ge 0$ for HyQMOM), which suffices to conclude realizability of $\overline{\bfM}^{n+1}_i$. Here, $A>B$ (resp.~$A\ge B$) denotes that the matrix $A-B$ is positive definite (resp.~semidefinite).

    First, if $\bfM_{i-\frac12}^{n,+}\in\Omega_2$, then by Lemma \ref{prop1} (Gaussian-EQMOM) and \eqref{eq:delta_tn_pm_1},
    \[
        \mathcal H_2(\bfF(\bfM_{i-\frac12}^{n,+})-\delta_{i-\frac12}^{n,-}\bfM_{i-\frac12}^{n,+})
        \ge \mathcal H_2(\bfF(\bfM_{i-\frac12}^{n,+})-\delta^-(\bfM_{i-\frac12}^{n,+})\bfM_{i-\frac12}^{n,+}) >0,
    \]
    implying $\mathcal H_2(\bfF(\bfM_{i-\frac12}^{n,+})) > \mathcal H_2(\delta_{i-\frac12}^{n,-}\bfM_{i-\frac12}^{n,+})$ and hence 
    \begin{equation} \label{eq:Pi1}
        \mathcal H_2(\Pi_1) > \left(\frac{1}{2}-\frac{\Delta t_n}{\Delta x}\left(\delta^{n,+}_{i-\frac12}-\delta^{n,-}_{i-\frac12}\right)	\right) \mathcal H_2(\mathbf{M}^{n,+}_{i-\frac12}) \ge 0,
    \end{equation}
    by the CFL condition \eqref{eq:cfl1}. The same argument applies to HyQMOM, and  The same arguments work for the HyQMOM, $\mathcal H_2(\Pi_2)$ is handled similarly for both models.

    Finally, $\mathcal H_2(\Pi_3)\ge 0$ and $\mathcal H_2(\Pi_4)\ge 0$ follow directly from Lemmas \ref{prop1} and \ref{prop2} as well as the facts $\delta_{i\pm\frac12}^{n,-}<0<\delta_{i\pm\frac12}^{n,+}$. This completes the proof.
\end{proof}

\subsection{Full realizability-preserving scheme} \label{sec3.3}
We now extend the scheme \eqref{1d_HLL_2ndScheme} to the full model \eqref{eq_mom_trunc}, including the source term \eqref{eq_bgk_mom_source}:
\begin{equation}\label{1d_HLL_2ndScheme_BGK}
    \overline{\bfM}^{n+1}_{i}=\overline{\bfM}^{n}_{i} -\frac{\Delta t_n}{\Delta x}\left(\widehat{\bfF}( {\bfM}^{n,-}_{i+\frac12}, {\bfM}^{n,+}_{i+\frac12}) - \widehat{\bfF}({\bfM}^{n,-}_{i-\frac12}, {\bfM}^{n,+}_{i-\frac12}) \right) +\Delta t_n\overline{\mathbf{S}}^n_i, 
\end{equation}
where the interface values $\{\bfM^{n,\mp}_{i\pm\frac12}\}$ and the numerical flux $\mathbf{\widehat{\bfF}}$ are defined as in \eqref{1d_HLL_2ndScheme}. The term $\overline{\mathbf{S}}^n_i$ is evaluated explicitly and can accommodate general sources; for the BGK-type collision operator \eqref{eq_bgk_mom_source}, however, a semi-implicit treatment is preferable (see \eqref{1d_HLL_2ndScheme_BGK_im}).

Recalling $\mathbf S = \mathbf S(\bfM)$ in the \textit{closed} model \eqref{eq_mom_trunc} and denoting $\mathbf S_{i\pm\frac12}^{n,\mp}:=\mathbf S(\bfM_{i\pm\frac12}^{n,\mp})$, we compute the cell average by the trapezoidal rule
\begin{equation}
    \overline{\mathbf{S}}^n_i=\frac12\left({\mathbf{S}}^{n,+}_{i-\frac12} + {\mathbf{S}}^{n,-}_{i+\frac12}\right).   
\end{equation}
We now state the realizability-preserving property of the full scheme \eqref{1d_HLL_2ndScheme_BGK}.
	
\begin{theorem}\label{eq:th2}
    If $\overline{\bfM}^{n}_i$ and the reconstructed values $\bfM^{n,\pm}_{i+\frac12}$ in \eqref{1d_HLL_2ndScheme_BGK} are strictly realizable, then the updated $\overline{\bfM}^{n+1}_i$ is strictly realizable for both Gaussian-EQMOM and HyQMOM under the CFL condition
    \begin{equation}\label{eq:cfl2} 
        \Delta t_n\max_i\left\{\frac{\delta^{n,+}_{i+\frac12} - \delta^{n,-}_{i+\frac12}}{\Delta x} +\frac{1}{\mu_i}\right\} < \frac12,\quad \forall i,
    \end{equation} 		
    where
    \begin{equation} \label{eq_mu_constraint}
        \mu_i := \min \left\{ \mu_*({\mathbf{M}}^{n,+}_{i-\frac12}, \tau), \mu_*({\mathbf{M}}^{n,-}_{i+\frac12}, \tau) \right\},
    \end{equation}
    and $\mu_*$ denotes the unique positive root of $\mathcal{P}(\mu)=a_0\mu^2+a_1\mu+a_2$ with coefficients $a_i=a_i(\bfM, \tau)$ given in \eqref{eq:abc}.
\end{theorem} 

\begin{proof}
    Analogous to \eqref{eq:M_Pi}, we rewrite the scheme \eqref{eq:cfl2} as $\overline{\bfM}^{n+1}_{i}= \widetilde{\Pi}_1+\widetilde{\Pi}_2+\Pi_3+\Pi_4$ with
    \[
        \widetilde{\Pi}_1={\Pi}_1+ \frac{\Delta t_n}{2}{\mathbf{S}}^{n,+}_{i-\frac12}, \quad
	\widetilde{\Pi}_2={\Pi}_2+ \frac{\Delta t_n}{2}{\mathbf{S}}^{n,-}_{i+\frac12},
    \]
    where $\Pi_1,\Pi_2,\Pi_3$, and $\Pi_4$ are as defined in \eqref{eq:def_Pi}. It suffices to verify $\mathcal H_2(\tilde \Pi_k)>0$ for $k=1,2$. Again, we use $A>B$ (resp.~$A\ge B$) to denote that the matrix $A-B$ is positive definite (resp.~semidefinite). 
    By Lemma \ref{pro3}, ${\mathbf{M}}^{n,+}_{i-\frac12} + \mu {\mathbf{S}}^{n,+}_{i-\frac12}$ is realizable for any $0<\mu \le \mu_*({\bfM}^{n,+}_{i-\frac12}, {\mathbf{S}}^{n,+}_{i-\frac12})$, which implies
    \[
        \mathcal H_2( {\mathbf{S}}^{n,+}_{i-\frac12}) \ge -\frac{1}{\mu_i} \mathcal H_2({\mathbf{M}}^{n,+}_{i-\frac12}).
    \]
    Combining this with \eqref{eq:Pi1} yields
    \begin{equation}
        \mathcal H_2(\widetilde{\Pi}_1) \ge \left(\frac{1}{2}-\frac{\Delta t_n}{\Delta x}\left(\delta^{n,+}_{i-\frac12}-\delta^{n,-}_{i-\frac12}\right) -\frac{\Delta t_n}{\mu_i} \right) \mathcal H_2(\mathbf{M}^{n,+}_{i-\frac12})>0,
    \end{equation}
    by \eqref{eq:cfl2}. The same argument applies to $\mathcal H_2(\tilde\Pi_2)$.
\end{proof}

\begin{remark}
	For the two-node Gaussian-EQMOM, the strict inequality ``$<$'' in \eqref{eq:cfl2} can be relaxed to ``$\le$''. 
\end{remark}

For the BGK-type collision \eqref{eq_bgk_mom_source}, a semi-implicit treatment of the source is often preferable, leading to the following scheme: 
\begin{equation} \label{1d_HLL_2ndScheme_BGK_im}
\begin{aligned}
    \overline{\bfM}^{n+1}_{i}=&\overline{\bfM}^{n}_{i} -\frac{\Delta t_n}{\Delta x}\left(\widehat{\bfF}( {\bfM}^{n,-}_{i+\frac12}, {\bfM}^{n,+}_{i+\frac12}) - \widehat{\bfF}({\bfM}^{n,-}_{i-\frac12}, {\bfM}^{n,+}_{i-\frac12}) \right) \\
    &+\frac{\Delta t_n}{\tau}\left(\rho{\bf \Delta}_i^n - \overline{\bfM}^{n+1}_{i} \right), 
\end{aligned}
\end{equation}
where ${\bf \Delta}=(\Delta_0(U,\theta),\ldots,\Delta_4(U,\theta))^\top$ is a strictly realizable moment vector. The realizability-preserving property of \eqref{1d_HLL_2ndScheme_BGK_im} follows immediately from Theorem~\ref{eq:th1}. 
\begin{corollary} \label{cor:real}
    If $\overline{\bfM}^{n}_i$ and the reconstructed values $\bfM^{n,\pm}_{i+\frac12}$ in \eqref{1d_HLL_2ndScheme_BGK_im} are strictly realizable, then the updated $\overline{\bfM}^{n+1}_i$ remains strictly realizable for both Gaussian-EQMOM and HyQMOM under the CFL condition \eqref{eq:cfl1}.
\end{corollary}

From a multiscale viewpoint, the semi-implicit treatment of the BGK source stabilizes stiff relaxation and inherits the collisionless CFL, a property aligned with asymptotic-preserving design principles.
 
An essential practical point for Theorems \ref{eq:th1}, \ref{eq:th2} and Corollary~\ref{cor:real} is to ensure the strict realizability of the reconstructed interfacial states $\mathbf{M}^{n,\pm}_{i\mp\frac12}$. Since this is not always guaranteed by \eqref{linear_M} with \eqref{eq:Albada}, we adopt the limiter
\begin{equation}\label{eq:limiter}
	\widetilde{\mathbf{M}}^{n,\pm}_{i\mp\frac12}=
	\theta^{\star}\left(\mathbf{M}^{n,\pm}_{i\mp\frac12} - \overline{\bfM}^{n}_{i}\right) + \overline{\bfM}^{n}_{i}, 
\end{equation}
and the modified, strictly realizable $\widetilde{\mathbf{M}}^{n,\pm}_{i\mp\frac12}$ is used as the interfacial value in \eqref{1d_HLL_2ndScheme_BGK}. The parameter $\theta^{\star}\in[0,1]$ in \eqref{eq:limiter} is uniquely determined by a root-finding algorithm such as a bisection procedure that, to a prescribed tolerance, identifies the largest $\theta^\pm$ such that $\widetilde{\mathbf{M}}^{n,\pm}_{i\mp\frac12}\in\Omega_2$. We then set $\theta^{\star}=\min\{\theta^{-},\theta^{+}\}$ to ensure \eqref{eq:Mn}. As shown numerically in the next section, this limiter is crucial for stability without compromising accuracy. 

The HLL-type flux and limiter are formulated so that realizability is enforced independently of the stiffness of the BGK source, making the scheme suitable for \emph{multiscale} applications with relaxation times $\tau$ ranging from order-one to vanishingly small.

We close this section with two remarks.
\begin{remark} \label{rem:high_order}
   	The scheme \eqref{1d_HLL_2ndScheme_BGK} and its realizability analysis extend naturally to higher-order space–time discretizations. For time integration, strong-stability-preserving Runge–Kutta methods are suitable because they can be expressed as convex combinations of forward Euler steps; by convexity of the strictly realizable moment set, realizability carries over. Similarly, higher-order spatial discretizations are admissible provided the cell average $\overline{\bfM}_i^n$ can be written as a convex combination of interfacial and interior nodal values \cite{ZS1,cui2024optimal}.  
\end{remark}

\begin{remark}
	Our approach generalizes to truncated moment systems \eqref{eq_mom_trunc} with a larger number of moments. For Gaussian-EQMOM with $N\ge 3$, Lemma~\ref{prop1} is expected to hold with suitably larger values of $a$. For HyQMOM, general closure frameworks have been established in \cite{fox2022hyqmom} for arbitrary $N\ge 3$ corresponding to approximating the NDF by linear combinations of Dirac delta functions \cite{zhang2024hyqmom}; hence the key realizability statement of Lemma~\ref{prop1} applies directly. A first-order realizability-preserving scheme with kinetic-based flux for HyQMOM was proposed in Theorem 3.7 of \cite{zhang2024hyqmom} and can also  be generalized to higher order within the present framework. 
\end{remark}

\section{Numerical tests}\label{sec4}
This section presents a series of numerical experiments that assess the accuracy and robustness of the proposed realizability-preserving finite volume method for the five-moment system \eqref{eq_mom_trunc}. The suite comprises a smooth accuracy test (Example~\ref{ex1:SmoothTest}) and four nonsmooth benchmarks (Examples~\ref{ex:1dRiemann}–\ref{ex:1dDoubleRare}): a Riemann problem, a shock tube problem, the Shu–Osher problem, and a double-rarefaction problem. 
To probe \emph{multiscale} performance, we report results for both $\tau=\infty$ (pure transport) and $\tau=0.05$ (stiff relaxation), demonstrating uniform realizability preservation across regimes.
 
In all tests, the one-dimensional computational domain is discretized into $N_x$ uniform grids. We employ the second-order SSP RK scheme for time discretization, with the time step chosen to satisfy the CFL conditions \eqref{eq:cfl1} and \eqref{eq:cfl2}.

Special care is taken with the moments. For the Gaussian-EQMOM, after computing the primitive variables $\mathbf{W}$, we recompute the conservative moment vector $\mathbf{M}$ via \eqref{eq_InvM_Gaussian} to eliminate the errors caused by moment inversion. In addition, the limiter \eqref{eq:limiter} is applied to the interfacial reconstructions in all test cases.
 
In all figures, ``Ref–G” and ``Ref–H” denote the reference solutions for the Gaussian-EQMOM and HyQMOM closures, respectively. These references are obtained by solving \eqref{eq_mom_trunc} on 20,480 uniform cells for Example~\ref{ex1:SmoothTest} and on 8,000 uniform cells for Examples~\ref{ex:1dRiemann}–\ref{ex:1dDoubleRare}.

\subsection{Accuracy test}
\begin{exmp}[Convergence test for a smooth problem]\label{ex1:SmoothTest}
	This first example validates the convergence rates of the proposed schemes for \eqref{eq_mom_trunc} with the two-node Gaussian-EQMOM and three-point HyQMOM closures. We set $\tau = 0.1$. The computational domain is $[-1,1]$ boundary conditions. The smooth initial data $\bfM(0,x)$ are specified separately for the two closures as follows:
\begin{itemize}[leftmargin=*] 
	\item  Gaussian-EQMOM: $\rho_1(0,x)=1+0.2\sin(\pi x)$, $\rho_2(0,x)=1+0.2\cos(\pi x)$, $v_1(0,x)=1$, $v_2(0,x)=-1$ and $\sigma(0,x)=1$;
	\item HyQMOM: $\rho_1(0,x)=1+0.2\sin(\pi x)$, $\rho_2(0,x)=1+0.2\cos(\pi x)$, $\rho_3(0,x)=1+0.2\sin(\pi x)$, $v_1(0,x)=1$, $v_3(0,x)=-1$ and $U(0,x)=1$.
\end{itemize}
We run the simulations up to time $t=0.01$ and measure the following two types of numerical errors for a scalar quantity $M$:
\begin{equation*}
    e_1(M)=\frac{1}{N_x}\sum_{i=1}^{N_x}\big| \overline{M}_i-\overline{M}_i^{ref} \big|, \quad
    e_2(M)=\frac{1}{N_x} \sqrt{\sum_{i=1}^{N_x}(\overline{M}_i-\overline{M}_i^{ref})^2},
\end{equation*}
where $\overline{M}_i$ and $\overline{M}_i^{ref}$ are the cell averages in cell $i$ of the numerical and reference solutions, respectively. As noted above, the reference solutions are computed on 20,480 uniform cells. 
Tables \ref{ex:1dTest} and \ref{ex:1dTest_2} report the errors and convergence rates of our realizability-preserving method  for the Gaussian-EQMOM and HyQMOM closures, respectively. The expected second-order accuracy is clearly observed for both models across grid refinements.  

\begin{table}[t]
\centering
\begin{threeparttable}
\caption{Example \ref{ex1:SmoothTest}. Errors and convergence rates of the second-order realizability-preserving scheme for the Gaussian-EQMOM closure model.}
	\label{ex:1dTest}
	{\begin{tabular}{c|cc|cc|cc|cc|cccc}
			\toprule 	 
			$N_x$&  $e_1({M_0})$&rate& $e_1({M_1})$&rate & $e_1({M_2})$&rate&\cr 
			\hline
			 10&	2.11E-03&    -- & 	1.40E-03&     -- &	5.51E-03&    -- \\ 
			 20&	4.67E-04&   2.17&	5.35E-04&   1.39&	1.07E-03&   2.37\\ 
			 40&	1.56E-04&   1.58&	2.23E-04&   1.26&	3.61E-04&   1.56\\ 
			 80&	3.68E-05&   2.08&	5.32E-05&   2.07&	8.84E-05&   2.03\\ 
			160&	9.12E-06&   2.01&	1.24E-05&   2.10&	1.93E-05&   2.20\\ 
			320&	1.91E-06&   2.26&	2.48E-06&   2.32&	3.64E-06&   2.40\\ 
			640&	3.55E-07&   2.43&	4.81E-07&   2.36&	6.77E-07&   2.43\\ 
			1280&	7.42E-08&   2.26&	8.47E-08&   2.51&	1.34E-07&   2.34\\ 
			2560&	1.58E-08&   2.23&	1.51E-08&   2.49&	2.54E-08&   2.40\\ 
			\hline
			$N_x$&  $e_2({M_0})$&rate& $e_2({M_1})$&rate & $e_2({M_2})$&rate\cr 
			\hline 
			 10&	 2.31E-03&    -- &   1.45E-03&    -- &	6.00E-03&    -- \\ 
			 20&	 6.00E-04&   2.17&	 6.15E-04&   1.39&	1.40E-03&   2.37\\ 
			 40&	 2.48E-04&   1.58&	 2.96E-04&   1.26&	5.85E-04&   1.56\\ 
			 80&	 7.66E-05&   2.08&	 8.97E-05&   2.07&	1.70E-04&   2.03\\ 
			160&	 2.20E-05&   2.01&	 2.49E-05&   2.10&	4.36E-05&   2.20\\ 
			320&	 5.64E-06&   2.26&	 6.25E-06&   2.32&	9.55E-06&   2.40\\ 
			640&	 1.30E-06&   2.43&	 1.35E-06&   2.36&	2.20E-06&   2.43\\ 
			1280&	 3.14E-07&   2.26&	 2.35E-07&   2.51&	3.83E-07&   2.34\\ 
			2560&	 8.51E-08&   2.23&	 4.12E-08&   2.49&	6.45E-08&   2.40\\ 
			\bottomrule
	\end{tabular}}
\end{threeparttable}
\end{table}

\begin{table}[t]
\centering
\caption{Example \ref{ex1:SmoothTest}. Errors and convergence rates of the second-order realizability-preserving scheme for the HyQMOM closure model.}
\label{ex:1dTest_2}
{\begin{tabular}{c|cc|cc|cc|cc|ccc}
		\toprule 	 
		$N_x$&  $e_1({M_0})$&rate& $e_1({M_1})$&rate & $e_1({M_2})$&rate\cr 
		\hline
		 10&	7.17E-04&    -- & 	3.30E-04&    -- & 	3.99E-04&    -- &	\\ 
		 20&	2.00E-04&   1.84&	1.55E-04&   1.09&	1.71E-04&   1.22&	\\ 
		 40&	6.91E-05&   1.53&	4.44E-05&   1.81&	4.41E-05&   1.96&	\\ 
		 80&	1.92E-05&   1.85&	1.19E-05&   1.91&	1.04E-05&   2.08&	\\ 
		160&	5.37E-06&   1.84&	3.17E-06&   1.90&	2.53E-06&   2.04&	\\ 
		320&	1.58E-06&   1.77&	9.17E-07&   1.79&	6.81E-07&   1.89&	\\ 
		640&	5.52E-07&   1.51&	2.91E-07&   1.66&	1.63E-07&   2.06&	\\ 
		1280&	1.86E-07&   1.57&	7.01E-08&   2.05&	3.15E-08&   2.37&	\\ 
		2560&	4.87E-08&   1.93&	1.29E-08&   2.44&	5.63E-09&   2.49&	\\ 
		\hline
		$N_x$&  $e_2({M_0})$&rate& $e_2({M_1})$&rate & $e_2({M_2})$&rate& \cr 
		\hline 
		 10&	8.53E-04&    --&	3.70E-04&    -- &	4.87E-04&    -- &	\\ 
		 20&	2.80E-04&   1.84&	2.55E-04&   1.09&	2.34E-04&   1.22&	\\ 
		 40&	1.23E-04&   1.53&	8.63E-05&    1.81&	7.61E-05&   1.96&	\\ 
		 80&	4.26E-05&   1.85&	2.95E-05&   1.91&	2.44E-05&   2.08&	\\ 
		160&	1.49E-05&   1.84&	1.04E-05&   1.90&	8.15E-06&   2.04&	\\ 
		320&	5.54E-06&   1.77&	4.00E-06&   1.79&	2.80E-06&   1.89&	\\ 
		640&	2.47E-06&   1.51&	1.55E-06&   1.66&	7.67E-07&   2.06&	\\ 
		1280&	1.13E-06&   1.57&	4.37E-07&   2.05&	1.53E-07&   2.37&	\\ 
		2560&	3.88E-07&   1.93&	7.98E-08&   2.44&	2.98E-08&   2.49&	\\  
		\bottomrule
\end{tabular}} 
\end{table}
\end{exmp} 

\subsection{Non-smooth problems}
\begin{exmp}[Riemann problem]\label{ex:1dRiemann}
	We consider a Riemann problem on the domain $[-1,1]$ with outflow boundary conditions. The initial data are
	\[
	\left(M_0,M_1,M_2,M_3,M_4\right)=
	\begin{cases}
		\left(1,\,1,\,\tfrac{4}{3},\,2,\,\tfrac{10}{3}\right), & x<0,\\[2pt]
		\left(1,\,-1,\,\tfrac{4}{3},\,-2,\,\tfrac{10}{3}\right), & x>0.
	\end{cases}
	\]
	We solve the model \eqref{eq_mom_trunc} with $\tau=\infty$ (no source term) and $\tau=0.05$, respectively. The solutions at $t=0.1$ for both closures are shown in Figures~\ref{Fig:1dRiemann_1} and \ref{Fig:1dRiemann_2}, respectively. The test case with $\tau=\infty$ was studied previously for the two-node Gaussian-EQMOM \cite{chalons2017multivariate} and three-point HyQMOM \cite{fox2018conditional} using a first-order scheme with kinetic flux; our second-order results in Figure~\ref{Fig:1dRiemann_1} are consistent with those findings. 
	As shown in Figures~\ref{Fig:1dRiemann_1}--\ref{Fig:1dRiemann_2}, our numerical solutions computed on 800 uniform cells closely match the reference solutions obtained on 8000 cells, highlighting the accuracy and robustness of the proposed scheme. Regarding the closures, Gaussian-EQMOM and HyQMOM yield noticeably different predictions on $\overline{M}_5$ for both $\tau$ values, while the lower-order moments---especially $M_0$, $M_1$, and $M_2$---agree closely.
	\begin{figure}[!htbp]
		\centering 
		\includegraphics[width=1.0\linewidth]{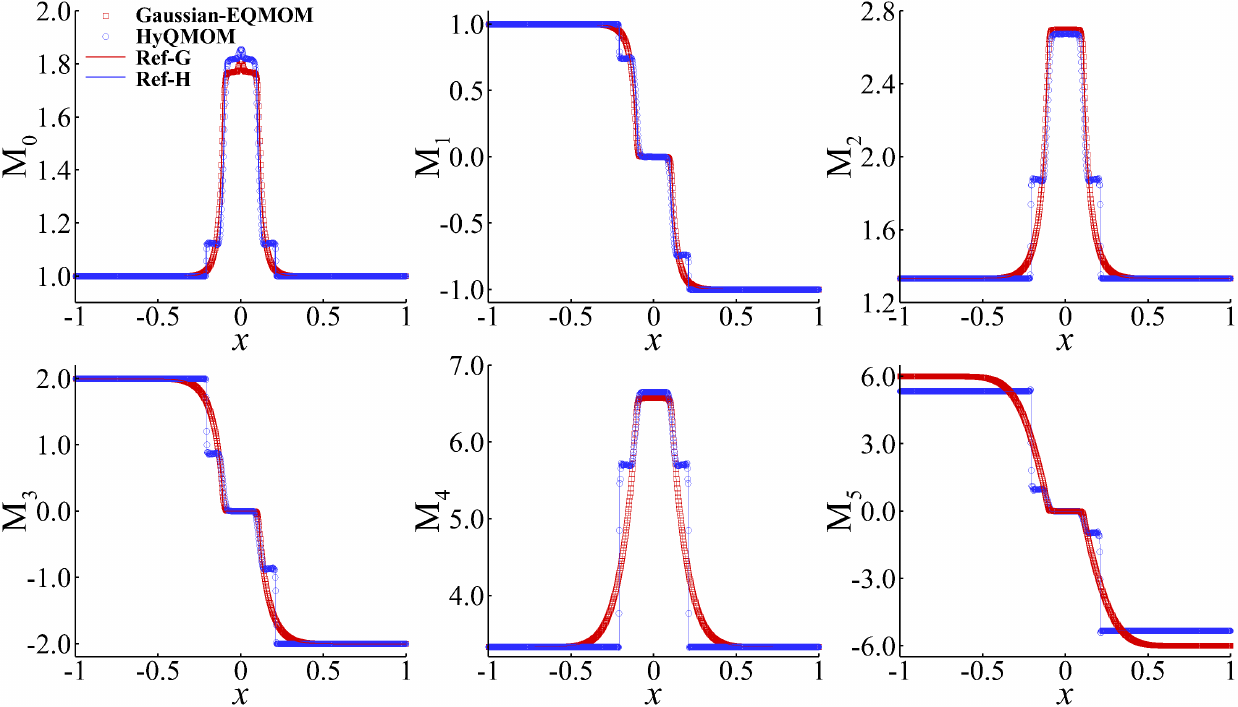}
		\caption{Computed moments $\{ M_\ell(t=1,x)\}^{5}_{\ell=0}$ for Example~\ref{ex:1dRiemann} with $\tau=\infty$ (no source term). 
		}\label{Fig:1dRiemann_1}
	\end{figure}  
	\begin{figure}[!htbp]
		\centering \includegraphics[width=1.0\linewidth]{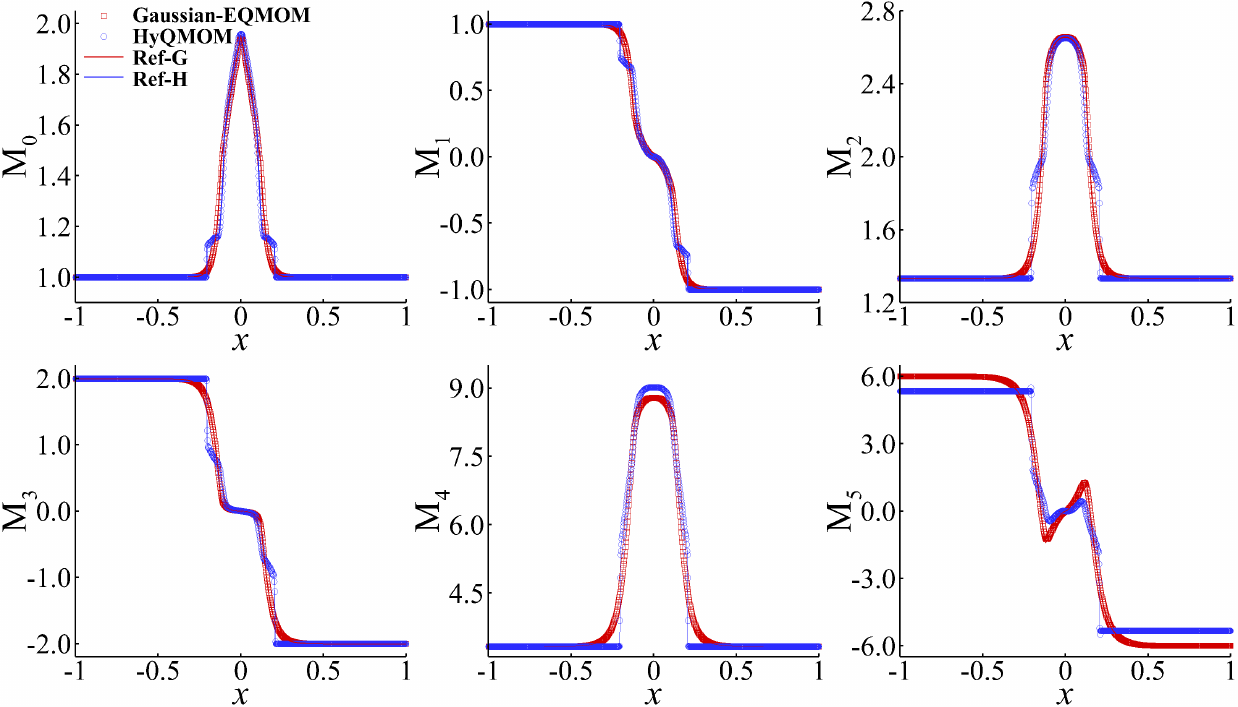}
		\caption{Computed  moments $\{M_\ell(t=1,x)\}^{5}_{\ell=0}$ for Example~\ref{ex:1dRiemann} with $\tau=0.05$.    
		}\label{Fig:1dRiemann_2}
	\end{figure}  
\end{exmp}

\begin{exmp}[Shock tube problem]\label{ex:1dShockTube}
	We consider another test with discontinuous initial data on the domain $[-1,1]$ with outflow boundary conditions. The model \eqref{eq_mom_trunc} is solved with $\tau=\infty$ (no source term). As in Example~\ref{ex1:SmoothTest}, the initial data $\bfM(0,x)$ are specified separately for the two closures:
	\[
	\text{Gaussian-EQMOM:}\quad
	(\rho_1,\rho_2,v_1,v_2,\sigma)=
	\begin{cases}
		(0.35,\,0.65,\,-1.5,\,2,\,0.8), & x<0,\\
		(0.09,\,0.01,\,-1.5,\,2,\,0.8), & x>0,
	\end{cases}
	\]
	\[
	\text{HyQMOM:}\quad
	(\rho_1,\rho_2,\rho_3,v_1,v_2,v_3)=
	\begin{cases}
		(0.2,\,0.35,\,0.45,\,-1.5,\,-1,\,2), & x<0,\\
		(0.05,\,0.01,\,0.04,\,-1.5,\,-1,\,2), & x>0.
	\end{cases}
	\]
	These choices ensure realizability of the initial data. Figures \ref{Fig:1dShockTube_1} and \ref{Fig:1dShockTube_2} shows the results at $t=0.2$ for $\tau=\infty$ and $\tau=0.05$, respectively, including the moments and several primitive variables $\mathbf W$. For both closures, the solutions computed on 800 uniform cells closely match the reference solutions on 8000 cells, confirming the accuracy and robustness of the proposed realizability-preserving schemes. Moreover, no spurious oscillations are observed.
\begin{figure}[!htbp]
	\centering 
	\includegraphics[width=1.0\linewidth]{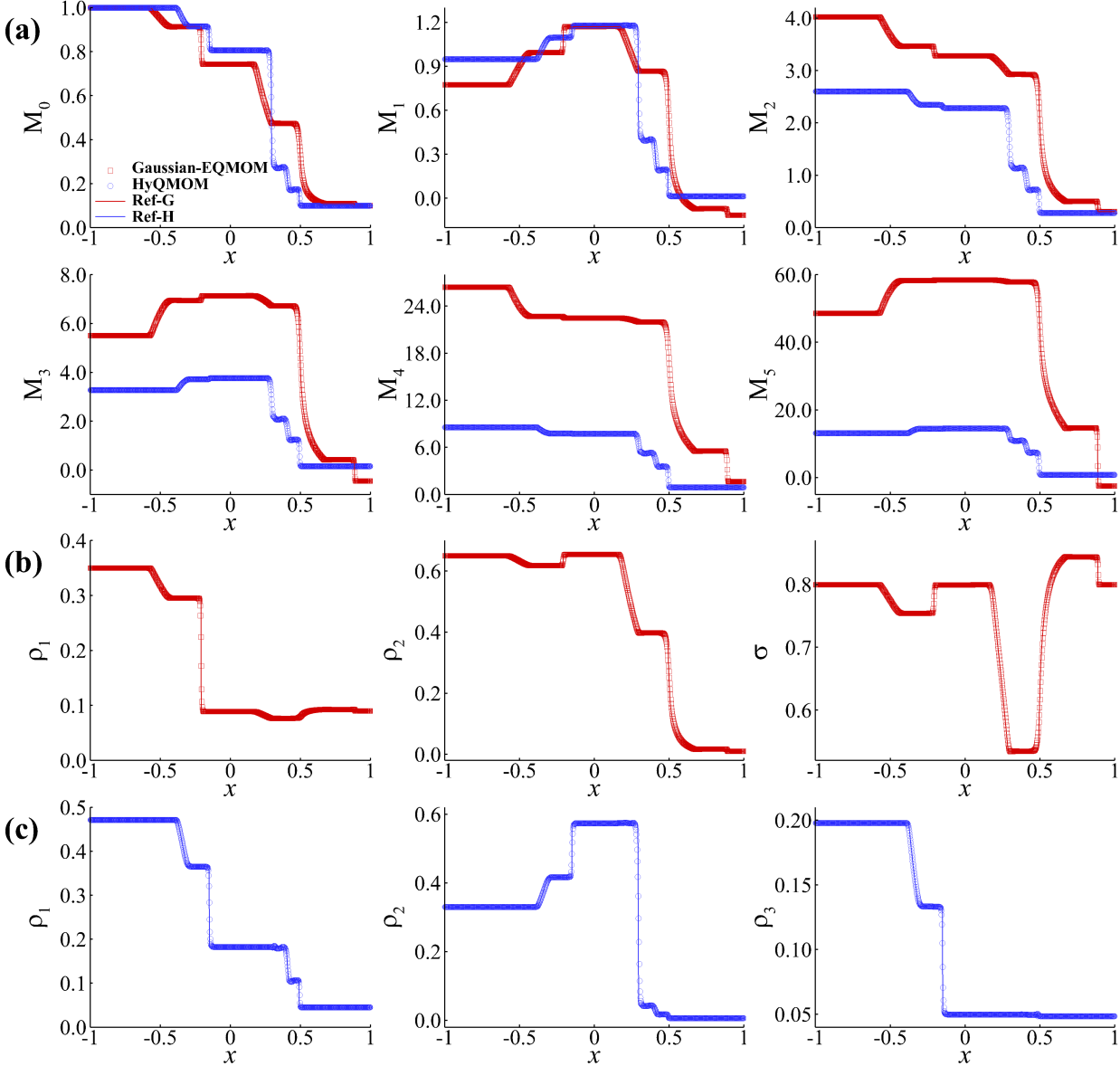}
	\caption{Numerical results for Example~\ref{ex:1dShockTube} at $t=0.2$ with $\tau=\infty$ (no source). 
		(a): computed moments $\{M_\ell(t=1,x)\}_{\ell=0}^{5}$ computed with the two closures (each using its own initial data). 
		(b): $\rho_1$, $\rho_2$, and $\sigma$ for the Gaussian-EQMOM system. 
		(c): $\rho_1$, $\rho_2$, and $\rho_3$ for the HyQMOM system.   
	} \label{Fig:1dShockTube_1}
\end{figure}

\begin{figure}[!htbp]
    \centering 
    \includegraphics[width=1.0\linewidth]{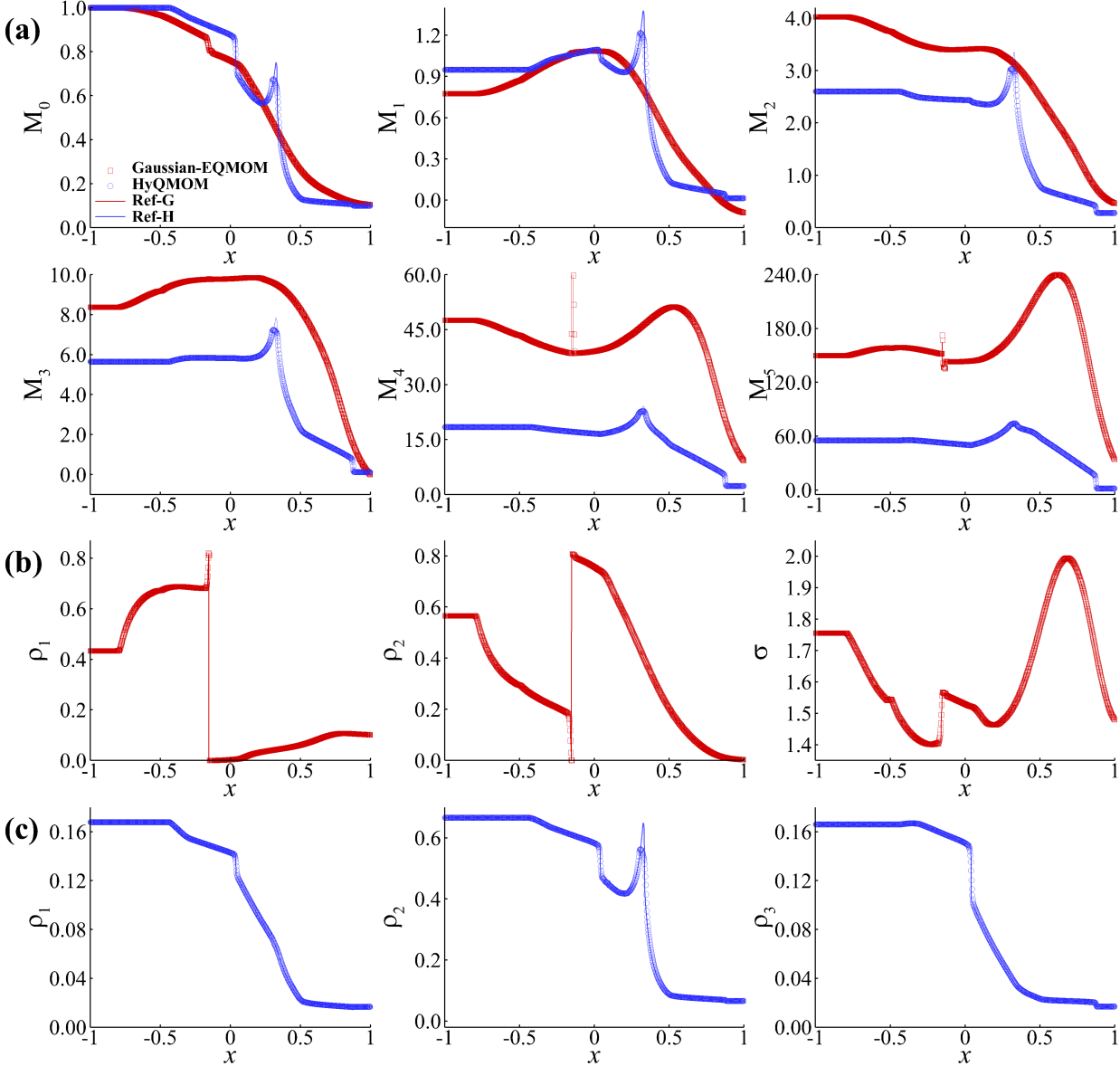}
       \caption{Numerical results for Example~\ref{ex:1dShockTube} at $t=0.2$ with $\tau=0.05$ . 
    	 (a): computed moments $\{M_\ell(t=1,x)\}_{\ell=0}^{5}$ computed with the two closures (each using its own initial data).  
    	 (b): $\rho_1$, $\rho_2$, and $\sigma$ for the Gaussian-EQMOM system. 
    	 (c): $\rho_1$, $\rho_2$, and $\rho_3$ for the HyQMOM system.  
    	 } \label{Fig:1dShockTube_2}
\end{figure}
\end{exmp}

\begin{exmp}[Shu--Osher problem]\label{ex:1dShuOsher}
	This example examines the interaction between a shock and high-frequency sinusoidal waves, extending the classical Shu--Osher problem for the compressible Euler equations \cite{shu1989efficient}. The computational domain is $[-5,5]$, with inflow and outflow boundary conditions imposed on the left and right boundaries, respectively. We consider \eqref{eq_mom_trunc} with $\tau=\infty$ (no source term) and $\tau=0.05$. As in the previous examples, the discontinuous initial data $\mathbf M(0,x)$ are specified separately for the two closures: 
    \begin{equation*}
        \left(\rho_1,\rho_2,v_1,v_2,\sigma\right)=
	\begin{cases}
	   (1,1,2.5,1,0.5),&x<-4,
	   \\
	   (1+0.1\sin(5x),1+0.2\sin(5x),-0.5,2,0.5),&x>-4
	\end{cases}
    \end{equation*}
    for the Gaussian-EQMOM closure system, and
    \begin{equation*}
        \left(\rho_1,\rho_2,\rho_3,v_1,v_3\right)=
	\begin{cases}
	   (1,1,1,2.5,1),&x<-4,
	   \\
	   (1+0.1\sin(5x),1+0.2\sin(5x),1+0.3\sin(5x),-0.5,2),&x>-4
	\end{cases}
    \end{equation*} 
    for the HyQMOM closure system with the velocity $v_2=0$. 
    The numerical solutions at $t=1$ are shown in Figures~\ref{Fig:1dShuOsher_1} and \ref{Fig:1dShuOsher_2} for $\tau=\infty$ and $\tau=0.05$, respectively. In both cases, the results computed by our realizability-preserving schemes on 800 uniform cells agree closely with the reference solutions on 8000 cells and exhibit no spurious oscillations. Moreover, the weights $\rho_i$ remain positive for both closure systems throughout the computation, demonstrating the effectiveness and robustness of the proposed schemes. 
    \begin{figure}[!htbp]
    	\centering  
    	\includegraphics[width=1.0\linewidth]{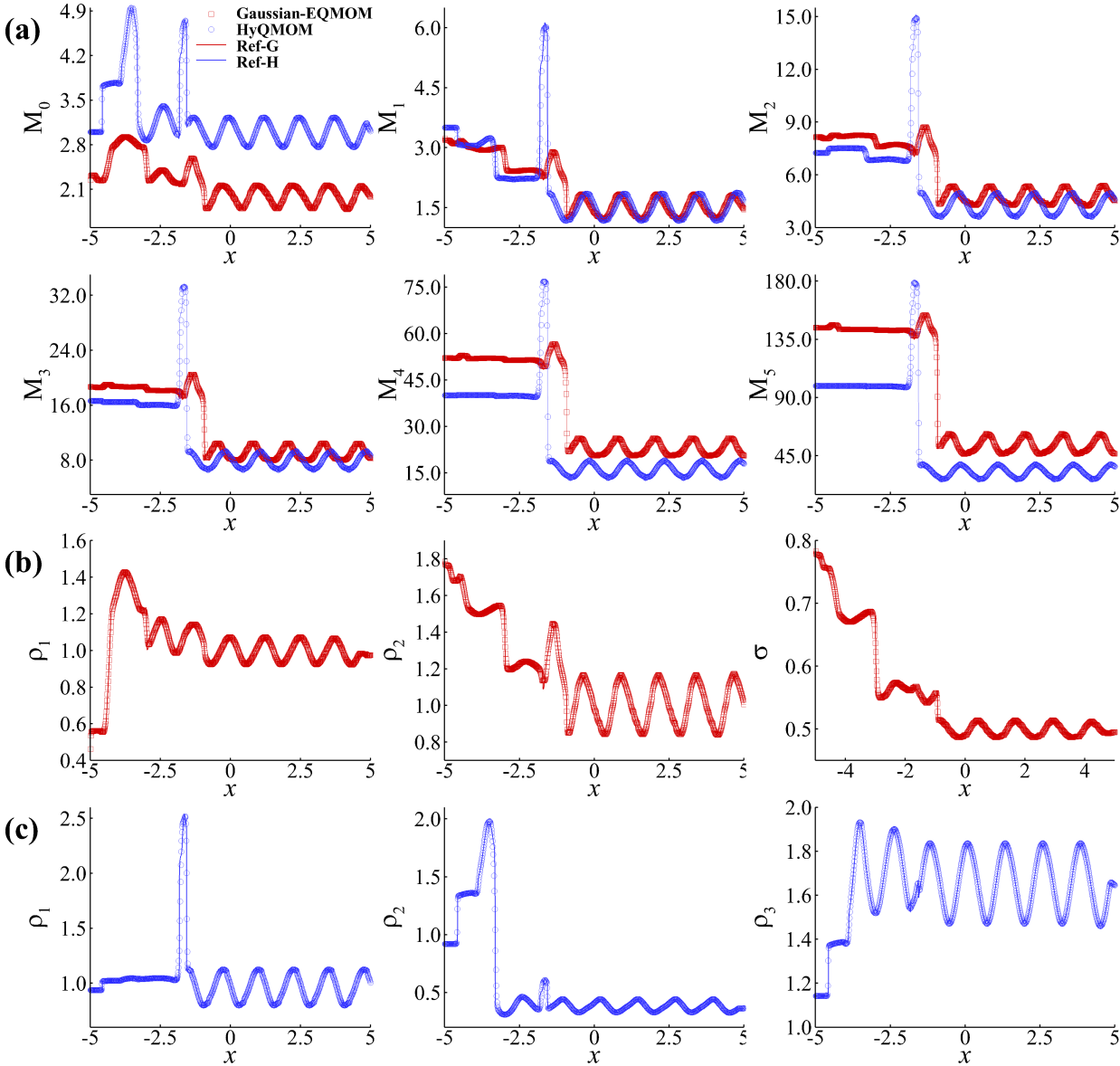}
    	\caption{Numerical results of Example~\ref{ex:1dShuOsher} at $t=1$ with $\tau=\infty$ (no source term). 
    		(a): computed moments $\{M_\ell(t=1,x)\}_{\ell=0}^{5}$ for the two closure systems (each with its own initial data). 
    		(b): $\rho_1$, $\rho_2$, and $\sigma$ for the Gaussian-EQMOM case. 
    		(c): $\rho_1$, $\rho_2$, and $\rho_3$ for the HyQMOM case. 
    	}
    	\label{Fig:1dShuOsher_1}
    \end{figure}
    
    \begin{figure}[!htbp]
    	\centering  
    	\includegraphics[width=1.0\linewidth]{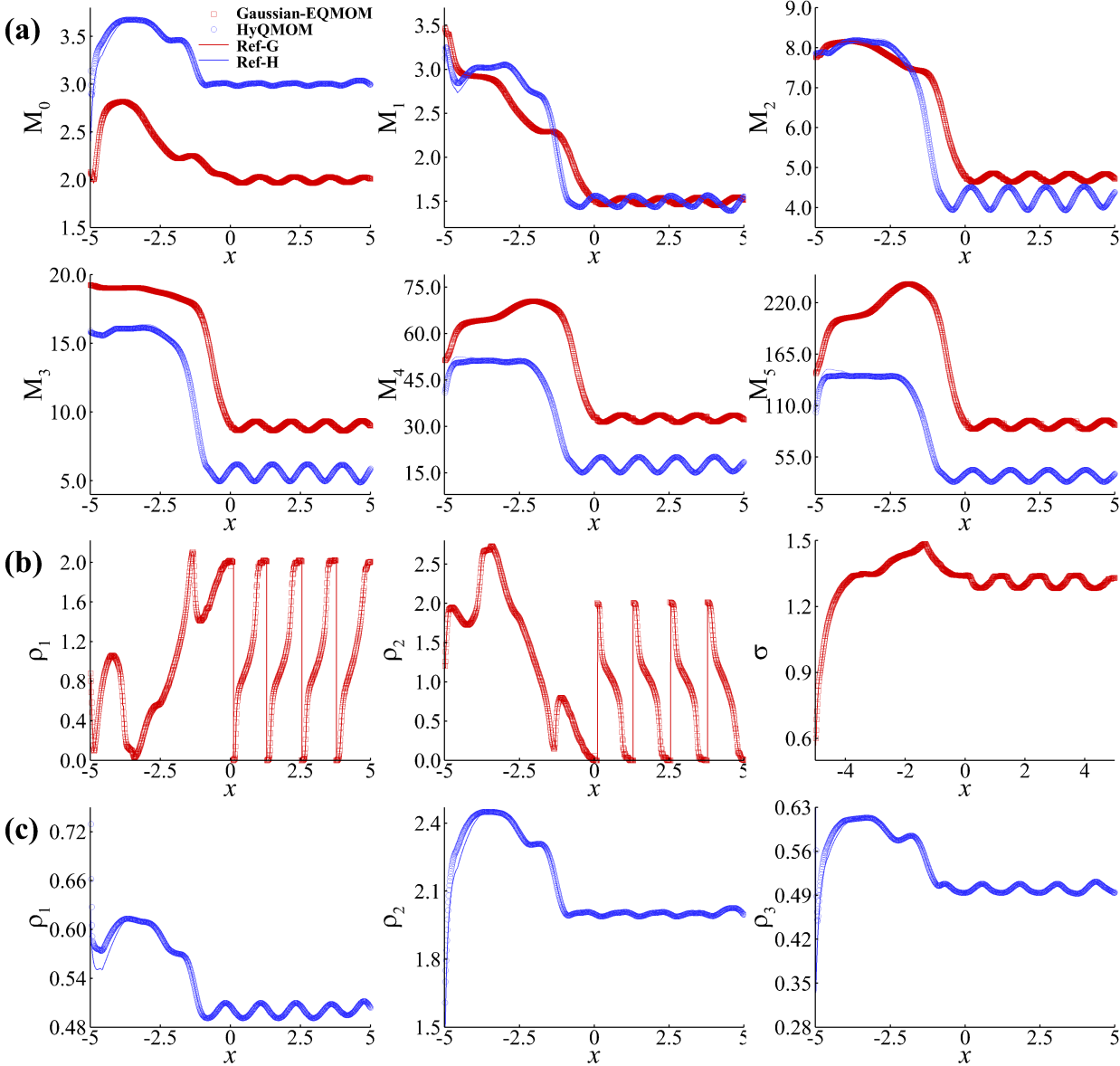}
    	\caption{Numerical results of Example~\ref{ex:1dShuOsher} at $t=1$ with $\tau=0.05$. 
    		(a): computed moments $\{M_\ell(t=1,x)\}_{\ell=0}^{5}$ for the two closure models (each with its own initial data).  
    		(b): $\rho_1$, $\rho_2$, and $\sigma$ for the Gaussian-EQMOM case.  
    		(c): $\rho_1$, $\rho_2$, and $\rho_3$ for the HyQMOM case. 
    	} \label{Fig:1dShuOsher_2}
    \end{figure}
\end{exmp}

\begin{exmp}[Double-rarefaction problem]\label{ex:1dDoubleRare}
	This Riemann problem involves two rarefaction waves on the domain $[-1,1]$ with discontinuous initial data and outflow boundary conditions. We solve the model \eqref{eq_mom_trunc} with $\tau=\infty$ (no source term) and with $\tau=0.05$, respectively. The initial conditions $\bfM(0,x)$ are specified separately for the two closure systems:
    \begin{equation*}
        \text{Gaussian-EQMOM:}\quad
        \left(\rho_1,\rho_2,v_1,v_2,\sigma\right)=
	\begin{cases}
            (0.5,0.5,-5,1,1),&x<0,
            \\
            (0.5,0.5,-1,5,1),&x>0,	
	\end{cases}
    \end{equation*}
    \begin{equation*}
        \text{HyQMOM:}\quad
	\left(\rho_1,\rho_2,\rho_3,v_1,v_2,v_3\right)=
	\begin{cases}
            (0.05,0.9,0.05,-5,-2,1),&x<0.
            \\
            (0.05,0.9,0.05,-1,2,5),&x>0,	   
	\end{cases}
    \end{equation*} 
This problem is challenging due to the presence of low-density regions. Using the realizability-preserving scheme, the results in Figures \ref{Fig:1dDoubleRare_1} and \ref{Fig:1dDoubleRare_2} demonstrate good robustness; in particular, the weights $\rho_i$ remain positive throughout this demanding test.

\begin{figure}[!htbp]
	\centering  
	\includegraphics[width=1.0\linewidth]{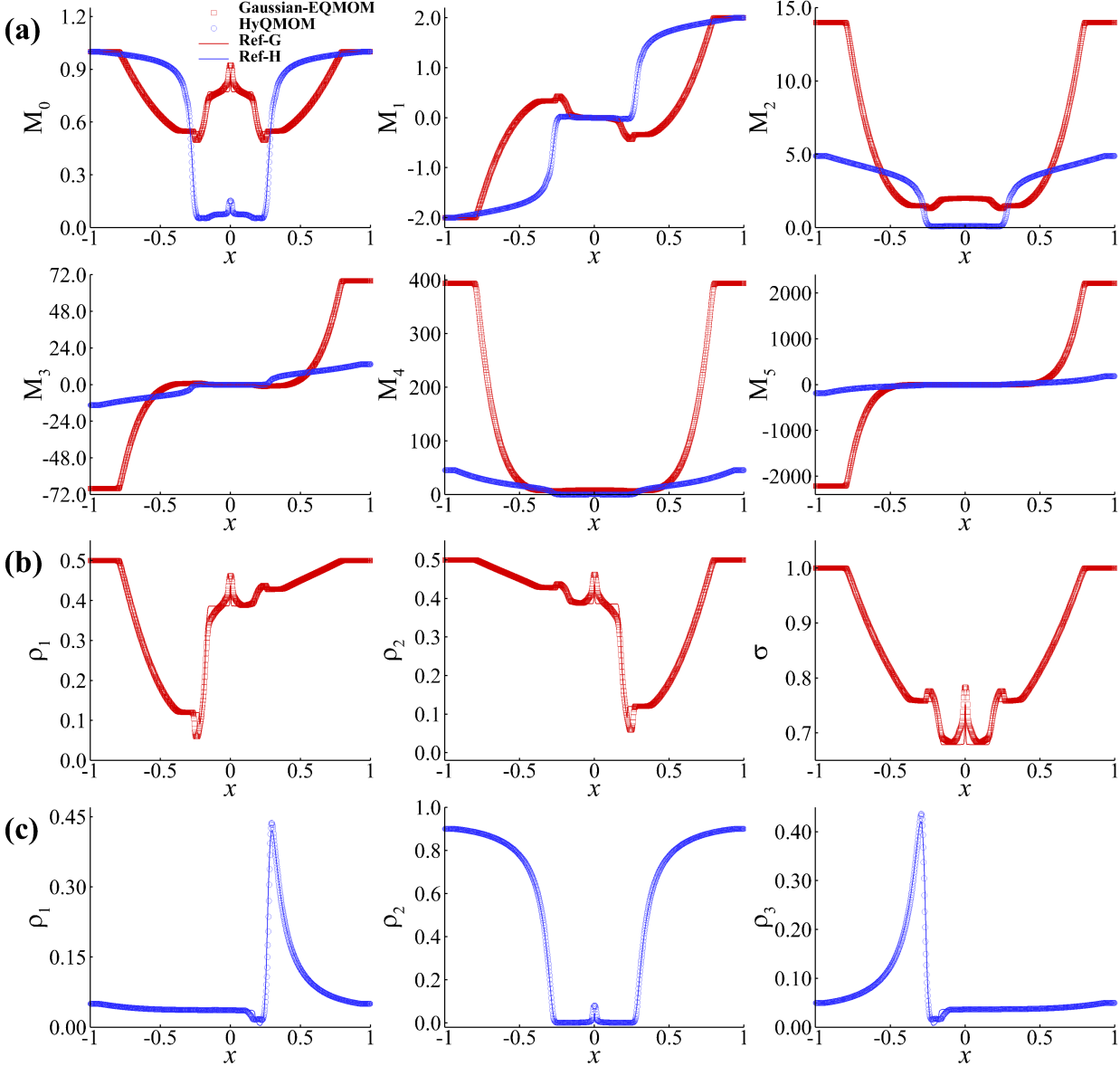}
	\caption{Numerical results of Example~\ref{ex:1dDoubleRare} at $t=0.12$ with $\tau=\infty$ (no source term). 
		(a): computed moments $\{M_\ell(t=1,x)\}_{\ell=0}^{5}$ for the two closure models (each with its own initial data). 
		(b): $\rho_1$, $\rho_2$, and $\sigma$ for the Gaussian-EQMOM case.  
		(c): $\rho_1$, $\rho_2$, and $\rho_3$ for the HyQMOM case. 
	}
	\label{Fig:1dDoubleRare_1}	
\end{figure}

\begin{figure}
	\centering  
	\includegraphics[width=1.0\linewidth]{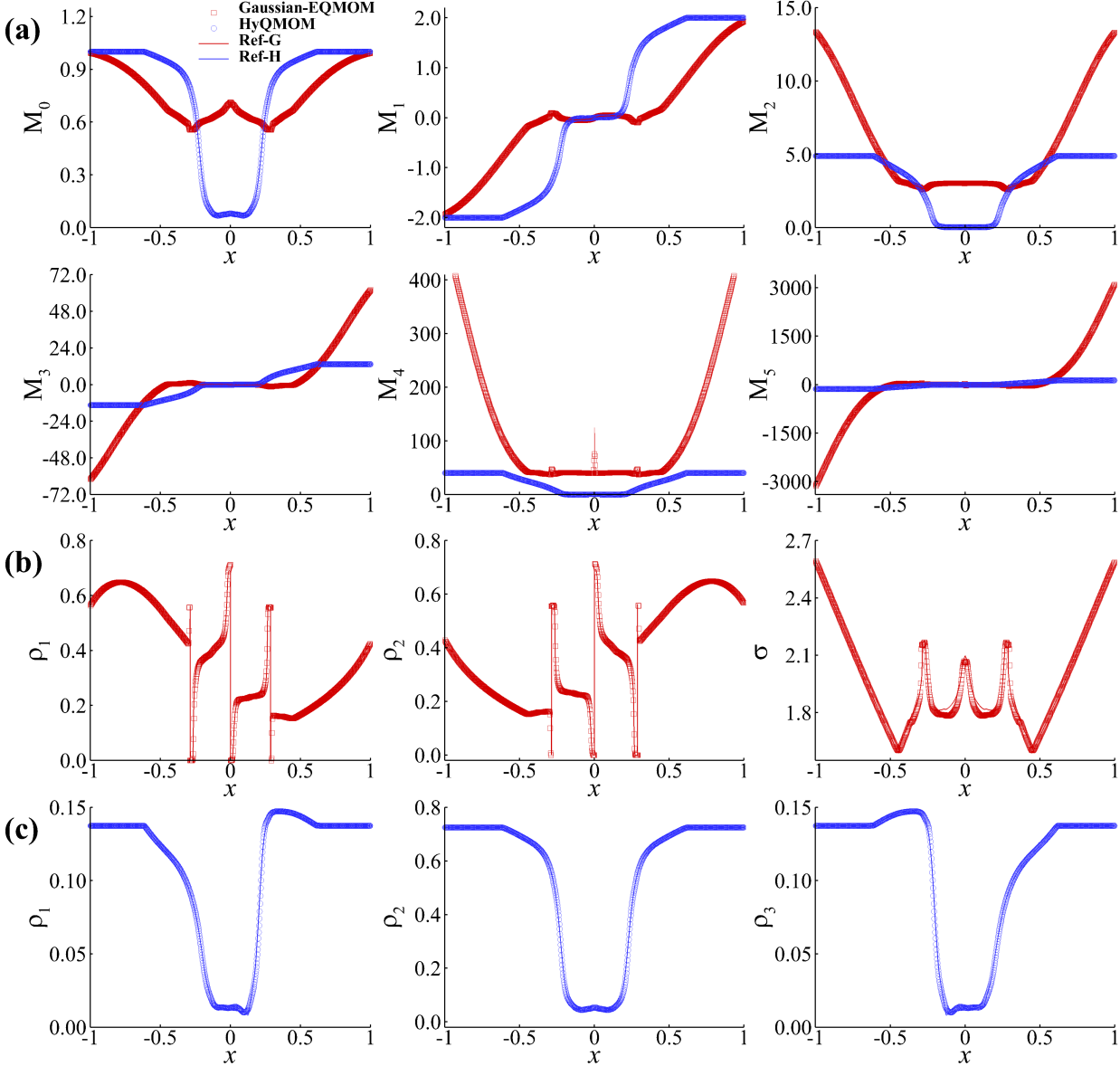}
	\caption{Numerical results of Example~\ref{ex:1dDoubleRare} at $t=0.12$ with $\tau=0.05$.
		(a): computed moments $\{M_\ell(t=1,x)\}_{\ell=0}^{5}$ for the two closure systems (each with its own initial data).  
		(b): $\rho_1$, $\rho_2$, and $\sigma$ for the Gaussian-EQMOM case. 
		(c): $\rho_1$, $\rho_2$, and $\rho_3$ for the HyQMOM case. 
	}
	\label{Fig:1dDoubleRare_2}
\end{figure} 
\end{exmp}

\section{Conclusion}\label{sec5} 

 We presented a provably realizability-preserving finite volume method for hyperbolic five-moment QBMM closures of kinetic equations, with a focus on the two-node Gaussian-EQMOM and three-point HyQMOM systems. It is based on a GQL-inspired reformulation that converts nonlinear realizability constraints into a nonnegative quadratic form in the moments, yielding bilinear inequalities enforced at the flux level via an HLL construction with closure-consistent wave-speed estimates. 
 From a multiscale standpoint, the theory provides realizability-preserving CFL conditions that remain meaningful across relaxation regimes and supports stiff source terms through a semi-implicit variant that inherits the collisionless CFL. The resulting schemes thus bridge kinetic and near-fluid regimes while maintaining positivity of QBMM weights and robustness in low-density regions. 
 Comprehensive experiments confirm the theory: we observe second-order convergence on smooth data and robust, oscillation-free resolution of discontinuities for both closures, including challenging low-density configurations. In all cases, lower-order moments produced by Gaussian-EQMOM and HyQMOM closely agree, and realizability is preserved.
 
The proposed framework extends naturally to higher-order space--time discretizations and to other QBMM variants. Future directions include extending to broader classes of moment-closure models, multidimensional formulations, and realizability-preserving discontinuous Galerkin and weighted essentially non-oscillatory (WENO) schemes tailored to multiscale kinetic--fluid applications.

\bibliography{reference}
\bibliographystyle{siamplain}

\end{document}